%% file: algebraicPoly_revised.tex
\documentclass[a4paper,twoside,11pt,english,intlimits]{article}

\usepackage[utf8]{inputenc}
\usepackage[T1]{fontenc}
\usepackage[english]{babel}

\usepackage{amsmath,amsthm,amsfonts,amssymb}
\usepackage{newtxtext, courier, euler}
\usepackage[scaled=0.95]{helvet}
\usepackage[cal=euler, scr=boondoxo, scrscaled=1.05]{mathalfa}

\usepackage{setspace}

\usepackage{euler,stmaryrd,upgreek}

\usepackage{enumerate}
\usepackage{fancyhdr,titlesec,url}
\usepackage[all,knot,poly]{xy}
\usepackage{array}
\usepackage{hyperref}
\usepackage{graphicx}
\usepackage[numbers,sort&compress]{natbib}
\usepackage{multirow}
\usepackage{ifthen}
\usepackage{time}
\usepackage{algorithm2e}
\usepackage{multicol}
\usepackage{etoolbox}
\usepackage{cancel}
\usepackage{tikz}

\newcommand{\AC}[1]{\poly{AC}^{(#1)}}
\newcommand{\Ass}[1]{\poly{Ass}^{(#1)}}
\newcommand{\Mag}[1]{\poly{Mag}^{(#1)}}
\newcommand{\Monn}[1]{\poly{Mon}^{(#1)}}
\newcommand{\Grpp}[1]{\poly{Grp}^{(#1)}}
\newcommand{\Abb}[1]{\poly{Ab}^{(#1)}}
\newcommand{\comm}[1]{C^{(#1)}}
\newcommand{\neutr}[1]{E^{(#1)}}
\newcommand{\inv}[1]{I^{(#1)}}
\newcommand{\grp}[1]{G^{(#1)}}
\newcommand{\asso}[1]{A^{(#1)}}

\input{macros}

\def\cart{\times}
\newcommand{\tha}[1]{\mathbb{#1}} 
\newcommand{\freetha}[1]{{#1}^{\cart}} 

\newcommand{\ob}[1]{\textsf{#1}}
\renewcommand{\id}{id}
\newcommand{\sort}[1]{\mathsf{\underline{#1}}}

\renewcommand{\int}[1]{\llbracket #1 \rrbracket}

\newcommand{\poly}[1]{\mathsf{#1}}

\renewcommand{\leq}{\leqslant}
\renewcommand{\geq}{\geqslant}

\definecolor{vert}{rgb}{0,0.45,0}
\definecolor{rouge}{rgb}{0.89,0.04,0.36}
\definecolor{MyGray}{gray}{0.6}
\definecolor{MyRed}{RGB}{212,42,42}

\renewcommand{\int}[1]{\llbracket #1 \rrbracket}
\newcommand{\pair}[1]{\langle #1 \rangle}
\newcommand{\boxx}{\boxempty}
\newcommand{\size}[1]{|#1|}

\def\N{\mathbb{N}}

\newcount\hh
\newcount\mm
\mm=\time
\hh=\time
\divide\hh by 60
\divide\mm by 60
\multiply\mm by 60
\mm=-\mm
\advance\mm by \time
\def\hhmm{\number\hh:\ifnum\mm<10{}0\fi\number\mm}

\newcommand{\st}{\partial}
\def\precdeg{\prec_{\text{deglex}}}

\newcommand{\er}[2]{\,_{#1}{#2}}
\newcommand{\re}[2]{{#2}_{#1}}
\newcommand{\ere}[2]{\,_{#1}{#2}_{#1}}

\newcommand{\PR}{\er{\poly{P}}{R}}
\newcommand{\RP}{\re{\poly{P}}{R}}
\newcommand{\PRP}{\ere{\poly{P}}{R}}

\def\dar[#1,#2,#3]{\ar@<0.8ex>[#1] ^{#3} \ar@<-0.8ex>[#1] _{#2} }

\newcommand{\Paux}{\Pr^{\text{db}}}

\newcommand{\auteur}[3]{
\noindent
\begin{minipage}[t]{.45\textwidth}
\begin{flushright}
\textsc{#1} \\
{\footnotesize\textsf{#2}}
\end{flushright} 
\end{minipage}
\qquad
\begin{minipage}[t]{.45\textwidth}
#3
\end{minipage}
}

\begin{document}
\thispagestyle{empty}

\begin{center}
\begin{doublespace}
\begin{huge}
{\scshape Confluence of algebraic rewriting systems}
\end{huge}

\vskip+2pt

\begin{huge}
{\scshape }
\end{huge}

\bigskip
\hrule height 1.5pt 
\bigskip

\vskip+5pt

\begin{Large}
{\scshape Cyrille Chenavier - Benjamin Dupont -- Philippe Malbos}
\end{Large}
\end{doublespace}

\vskip+50pt

\begin{small}\begin{minipage}{14cm}
\noindent\textbf{Abstract --}
Convergent rewriting systems on algebraic structures give methods to solve decision problems, to prove coherence results, and to compute homological invariants. These methods are based on higher-dimensional extensions of the critical branching lemma that proves local confluence from confluence of the critical branchings. The analysis of local confluence of rewriting systems on algebraic structures, such as groups or linear algebras, is complicated because of the underlying algebraic axioms. This article introduces the structure of algebraic polygraph modulo that formalizes the interaction between the rules of an algebraic rewriting system and the inherent algebraic axioms, and we show a critical branching lemma for algebraic polygraphs. We deduce a critical branching lemma for rewriting systems on algebraic models whose axioms are specified by convergent modulo rewriting systems. We illustrate our constructions for string, linear, and group rewriting systems.

\medskip


\medskip

\smallskip\noindent\textbf{M.S.C. 2010 -- Primary:} 68Q42, 18C10.
\textbf{Secondary:} 16S36, 13P10.
\end{minipage}\end{small}
\end{center}

\vskip+0pt

\section{Introduction}

\subsubsection*{Completion procedures}
The critical-pair completion (CPC) is an approach developed in the mid sixties that combines completion procedures and the notion of \emph{critical pair}, also called \emph{critical branching}~\cite{Shirshov62,Bergman78,Buchberger87}. It originates from theorem proving~\cite{Robinson65}, polynomial ideal theory~\cite{Janet20,Buchberger65}, word problem in algebras~\cite{KnuthBendix70,Nivat73,LeChenadec84}, and has found many applications to solve algorithmic problems, see \cite{Buchberger87,IoharaMalbos20} for an historical account.
In the mid eighties CPC has found original and deep applications in algebra in order to solve coherence problems for monoids~\cite{Squier94,GuiraudMalbos18}, and monoidal categories~\cite{GuiraudMalbos12mscs, CurienMimram17}, or to compute homological invariants of associative algebras \cite{Anick86}, and monoids \cite{Squier87,Kobayashi90}. The CPC was extended to two-dimensional rewriting systems in \cite{Mimram10,GuiraudMalbos09}.
More recently, higher-dimensional extensions of the CPC were applied to the computation of free resolutions and cofibrant replacements of algebraic and categorical structures~\cite{GuiraudMalbos12advances, GaussentGuiraudMalbos15, MimramMalbos16,GuiraudHoffbeckMalbos19} and operads~\cite{MalbosRen20, MalbosRen21}. 
The obstructions in each dimension are formulated in terms of critical branchings. While generators and rules are in dimension $1$ and $2$ respectively, the critical branchings, and the critical triple branchings, that is overlappings of rules on critical branchings, describe $3$-dimensional and $4$-dimensional cocycles respectively. This generalizes in higher-dimensions, where for $n\geq 4$, the $n$-dimensional cocycles are described by overlappings of a rule on a critical $(n-1)$-branching. These constructions based on CPC are known for monoids, small categories, and algebras. However, the extension to a wide range of algebraic structures is complicated due to the interaction between the rewriting rules and the inherent axioms of the algebraic structure. For this reason, the higher-dimensional extensions of the CPC for a wide range of algebraic structures, including groups, Lie rings, is still an open problem.

\subsubsection*{Critical branching lemma}
One of the main tools to reach confluence in CPC procedures for algebraic rewriting systems is the \emph{critical branching lemma}, by Knuth-Bendix, \cite{KnuthBendix70}, and Nivat, \cite{Nivat73}. Nivat showed that the local confluence of a string rewriting system (SRS) is decidable, whether it is terminating or not.
The proof is based on classification of the local branchings into \emph{orthogonal} branchings, that involve two rules that do not overlap, and \emph{overlapping} branchings. A \emph{critical branching} is a minimal overlapping application of two rules on the same redex.
When the orthogonal branchings are confluent, if all critical branchings are confluent, then local confluence holds.
Thus, the main argument to achieve critical branching lemma is to prove that orthogonal and overlapping branchings are confluent.
For SRS and term rewriting systems (TRS), orthogonal branchings are always confluent, and confluence of critical branchings implies confluence of overlapping branchings. The situation is more complicated for rewriting systems on a linear structure.

The well known approaches of rewriting in the linear context consist in orienting the rules with respect to an ambiant monomial order, and critical branching lemma is well known in this context. However, 
some algebras do not admit any higher-dimensional finite convergent presentation on a fixed set of generators with respect to a monomial order, \cite{GuiraudHoffbeckMalbos19}.
Due to algebraic perspectives, an approach of linear rewriting where the orientation of rules does not depend of a monomial order was introduced in~\cite{GuiraudHoffbeckMalbos19}. 
However, in that setting there are two conditions to guarantee a critical branching lemma, namely termination and positivity of reductions. A positive reduction for a linear rewriting system (LRS), as defined in~\cite{GuiraudHoffbeckMalbos19}, is the application of a reduction rule on a monomial that does not appear in the polynomial context.
For instance, consider the LRS on an associative algebra given in \cite{GuiraudHoffbeckMalbos19} defined by the following two rules 
\[
\alpha: xy\fl xz,
\qquad
\beta : zt \fl 2yt.
\]
It has no critical branching, but it has the following non-confluent additive branching:

\vskip-10pt

\[
\xymatrix @R=0.18em {
&& 4xyt 
	\ar [r] ^-*+{4\alpha t}
& 4xzt
	\ar [r] ^-*+{4x\beta}
& \cdots
\\
& 2xzt
	\ar@/^/ [ur] ^-*+{2x\beta}
	\ar@{.>} [dr] ^(0.55)*+{xzt + x\beta}
\\
xyt + xzt 
	\ar@/^/ [ur] ^-*+{\alpha t + xzt}
	\ar@/_/ [dr] _-*+{xyt + x\beta}
	\ar@{} [rr] |-{ =}
&& xzt + 2xyt
\\
& 3xyt
	\ar@{.>} [ur] _(0.55)*+{\alpha t + 2xyt}
	\ar@/_/ [dr] _-*+{3\alpha t}
\\
&& 3xzt
	\ar [r] _-*+{3x\beta}
& 6xyt
	\ar [r] _-*+{6\alpha t}
& \cdots
}
\]
The dotted arrows correspond to non positive reductions. This example illustrates that the lack of termination is an obstruction to confluence of orthogonal branchings in a \emph{left-monomial LRS}, that is whose rules transform a monomial into a polynomial. Indeed, the critical branching lemma for linear $2$-dimensional polygraphs states that a terminating left-monomial linear polygraph is locally confluent if and only if all its critical branchings are confluent, {\cite[Theorem 4.2.1]{GuiraudHoffbeckMalbos19}}

\subsubsection*{Rewriting modulo}
Rewriting modulo appears naturally in algebraic rewriting when studied reductions are defined modulo the axioms of an ambiant algebraic or categorical structure, \emph{e.g.} rewriting in commutative, groupoidal, linear, pivotal, weak structures. Furthermore, rewriting modulo facilitates the analysis of confluence. In particular, rewriting modulo a set of relations makes the property of confluence easier to prove. Indeed, the family of critical branchings that should be considered in the analysis of confluence is reduced, and the non-orientation of a part of the relations allows more flexibility when reaching confluence.

The most naive approach of rewriting modulo is to consider the rewriting system $\PRP$ consisting in rewriting on congruence classes modulo the axioms $\poly{P}$. This approach works for some equational theories, such as associative and commutative theories. However, it appears inefficient in general for the analysis of confluence. Indeed, the reducibility of an equivalence class needs to explore all the class, hence it requires all equivalence classes to be finite.
Another approach of rewriting modulo has been considered by Huet in \cite{Huet80}, where rewriting sequences involve only oriented rules and no equivalence steps, and the confluence property is formulated modulo equivalence. However, for algebraic rewriting systems such rewriting modulo is  too restrictive for computations, see~\cite{JouannaudLi12}. 
Peterson and Stickel introduced in~\cite{PetersonStickel81} an extension of Knuth-Bendix's completion procedure,~\cite{KnuthBendix70}, to reach confluence of a rewriting system modulo an equational theory, for which a finite, complete unification algorithm is known. They applied their procedure to rewriting systems modulo axioms of associativity and commutativity, in order to rewrite in free commutative groups, commutative unitary rings, and distributive lattices.
Jouannaud and Kirchner enlarged this approach in \cite{JouannaudKirchner84} with the definition of rewriting properties for any rewriting system modulo~$S$ such that $R\subseteq S \subseteq \PRP$.
They also proved a critical branching lemma and developed a completion procedure for rewriting systems modulo $\PR$, whose one-step reductions consist in application of a rule in $R$ using $\poly{P}$-matching at the source. Their completion procedure is based on a finite $\poly{P}$-unification algorithm. Bachmair and Dershowitz in \cite{BachmairDershowitz89} developed a generalisation of Jouannaud-Kirchner's completion procedure using inference rules. Several other approaches have also been studied for TRS modulo to deal with various equational theories, see \cite{MR778047, Viry95, marche1993, Marche96}.

\subsubsection*{Algebraic and categorical rewriting}
In this article, we use the notion of cartesian polygraphs as categorical models of TRSs introduced in \cite{MalbosMimram21} to formulate our constructions and prove our results. The polygraphic language provides a unified categorical framework for algebraic rewriting paradigms: abstract, string, term, linear rewriting and their higher-dimensional versions. Polygraphs also provide a natural setting to formulate higher-dimensional rewriting concepts such as coherence, that is two-dimensional word problems~\cite{Squier94, Lafont95, GuiraudMalbos18, curien2021coherent}, and normalisation strategies as rewriting tools to prove homotopical properties in higher algebra theory, \cite{GaussentGuiraudMalbos15,GuiraudMalbos12advances}. 
In Section~\ref{S:Preliminaries}, we recall the notion of cartesian $2$-dimensional polygraphs introduced in \cite{MalbosMimram21} as categorical interpretations of TRS and presentations of Lawvere algebraic theories. A \emph{cartesian $2$-polygraph} is defined by an equational signature $(\poly{P}_0,\poly{P}_1)$ and a cellular extension $\poly{P}_2$ of the free algebraic theory $\freetha{\poly{P}}_1$ on $(\poly{P}_0,\poly{P}_1)$. A rewriting path corresponds to a $2$-cell in the free algebraic $2$-theory generated by the $2$-polygraph $(\poly{P}_0,\poly{P}_1,\poly{P}_2)$.

\subsubsection*{Algebraic polygraphs}
In Section~\ref{S:AlgebraicPolygraphsModulo}, we introduce a categorical model for rewriting in algebraic structures which formalizes the interaction between the rules of the rewriting system and the inherent axioms of the algebraic structure.
We define the structure of \emph{algebraic polygraph} as a data $(\poly{P},Q,R)$ made of a cartesian $2$-polygraph $\poly{P}$ and a set $Q$ of generating ground terms and a cellular extension $R$ on the ground terms. In Section~\ref{S:AlgebraicPolygraphs}, we introduce a notion of \emph{positive reduction strategy} on an algebraic polygraph in order to select admissible rewriting steps used to formulate rewriting properties modulo. 
The idea is to avoid termination and confluence obstructions from the underlying axioms for the quotiented algebraic rewriting system defined as a projection of the positive reductions in Section~\ref{SS:AlgebraicRewrSystem}.

\subsubsection*{Algebraic critical branching lemma}
Following \cite{DupontMalbos18b}, in Section~\ref{SS:AlgebraicPolygraphModulo} we define  the structure of algebraic polygraph modulo as a data $\Pr=(\poly{P},Q,R,S)$ made of an algebraic polygraph $(\poly{P},Q,R)$ and a cellular extension $S$ on the ground terms, and that depends on the cellular extension $R$ and the algebraic axioms of $\poly{P}_2$. As a consequence, the rewriting properties of $\Pr$ depend on the interaction between the rules of the rewriting system and the inherent axioms of the algebraic structure. In Section~\ref{S:ConfluenceAlgebraicPolygraphsModulo}, we prove the Newman lemma for quasi-terminating algebraic polygraphs modulo, stated as follows:

\begin{quote}
{\bf Theorem \ref{T:NewmanLemmaModulo}.}
\emph{
Let $\Pr$ be a quasi-terminating algebraic polygraph modulo, and $\sigma$ be a positive strategy on $\Pr$. If $\Pr$ is locally $\sigma$-confluent modulo, then it is $\sigma$-confluent modulo.
}
\end{quote}

\noindent Then we prove a critical branching lemma for quasi-terminating algebraic polygraphs modulo. 

\begin{quote}
{\bf Theorem \ref{T:AlgebraicCriticalBranchingTheorem}.}
\emph{
Let $\Pr=(\poly{P},Q,R,S)$ be an algebraic polygraph modulo with a positive confluent strategy $\sigma$. If $\PRP$ is quasi-terminating, then an algebraic rewriting system on $\Pr$ is locally confluent if, and only if, its critical branchings are confluent.
}
\end{quote}

\noindent We deduce from this result a critical branching lemma for rewriting systems on algebraic structures, whose axioms are specified by TRS satisfying appropriate convergence properties modulo AC. Finally, we apply the above results to the linear rewriting setting. In particular, we explain why termination is a necessary condition to characterize local confluence in that case.

\medskip

\subsubsection*{Convention and notations}
An \emph{abstract rewriting system} (ARS) is a data $(X,R)$ made of a set $X$ and a set $R$ equipped with source and target maps $\partial^-,\partial^+: R \to X$ called a \emph{cellular extension} of $X$. An element $r$ of $R$ is denoted by $r_- \fl r_+$, where $r_-:=\partial^-(r)$ and $r_+:=\partial^+(r)$. We say that $r$ composes with $r'$ if $\partial^+(r) = \partial^-(r')$. We denote by $\overset{\ast}{\fl}$ the symmetric, transitive closure of $\fl$ with respect to this composition. We say that $x$ \emph{rewrites into} $y$ if $x \overset{\ast}{\fl} y$.

The ARS $(X,R)$ is \emph{terminating} (resp. \emph{quasi-terminating}) if there is no sequence $(x_n)_{n \in \Nb}$ such that $x_n \fl x_{n+1}$ (resp. if for each sequence $(x_n)_{n \in \Nb}$  such that $x_n \fl x_{n+1}$, the sequence $(x_n)_{n \in \Nb}$ contains an infinite number of occurrences of the same element). It is \emph{confluent} if, whenever $x \overset{\ast}{\fl} y$ and $x \overset{\ast}{\fl} z$, there exists $t$ such that $y \overset{\ast}{\fl} t$ and $z \overset{\ast}{\fl} t$. An element $x$ of $X$ is called a \emph{normal form} for $(X,R)$ if there is no $y$ such that $x \fl y$. Given an equivalence relation $\equiv$ on $X$, we say that $(X,R)$ is \emph{confluent modulo $\equiv$} if, whenever $x \equiv y$ and $x \overset{\ast}{\fl} x'$, $y \overset{\ast}{\fl} y'$, there exist $z, z' \in X$ such that $x' \overset{\ast}{\fl} z$, $y' \overset{\ast}{\fl} z'$, and $z \equiv z'$. 

\section{Preliminaries on algebraic theories}
\label{S:Preliminaries}

In this section we recall notions on algebraic theories from \cite{Lawvere63} and the structure of cartesian polygraph introduced in~\cite{MalbosMimram21} as a categorical model of term rewriting systems.

\subsection{Cartesian polygraphs and theories}
\label{SS:CartesianPolygraphsTheories}

\subsubsection{Signature and terms}
\label{SSS:SignatureAndTerms}
A \emph{signature} on a set $\poly{P}_0$ of \emph{sorts} is a directed graph
\[
\xymatrix{
\poly{P}_0^\ast
&
\ar@<+0.5ex>[l]^-{\partial_0^+}
\ar@<-0.5ex>[l]_-{\partial_0^-}
\poly{P}_1
}
\]
on the free monoid $\poly{P}_0^\ast$ over $\poly{P}_0$.
From a higher-dimensional rewriting approach, the data $(\poly{P}_0, \poly{P}_1)$ is called a \emph{$1$-polygraph}. 
An element $\alpha$ of $\poly{P}_1$ is called an \emph{operation}, and its source $\partial_0^-(\alpha) \in \poly{P}_0^\ast$ is called its \emph{arity} and its target $\partial_0^+(\alpha) \in \poly{P}_0$ its \emph{coarity}. For sorts $\ob{s}_1,\ldots,\ob{s}_k$, we denote $\sort{s}=\ob{s}_1\ldots \ob{s}_k$ their product in the free monoid $\poly{P}_0^\ast$. We denote $\size{\sort{s}}=k$ the \emph{length} of $\sort{s}$ and the sort $\ob{s}_i$ in $\sort{s}$ will be denoted by $\sort{s}_i$, so that $\sort{s}_i \in \poly{P}_0$.

Recall from \cite{Lawvere63} that a \emph{(multityped Lawvere algebraic) theory} on a set $\poly{P}_0$ of sorts is a category with finite products $\tha{T}$ together with a map $\iota : \poly{P}_0 \fl \tha{T}_0$, where $\tha{T}_0$ denotes the set of $0$-cells, and such that every $0$-cell in $\tha{T}_0$ is isomorphic to a finite product of $0$-cells in $\iota(\poly{P}_0)$. 
We denote by $\freetha{\poly{P}}_1$ the \emph{free theory generated} by a signature $(\poly{P}_0,\poly{P}_1)$. 
Its products on $0$-cells are induced by products of sorts in $\poly{P}_0^\ast$, and its $1$-cells are \emph{terms} over $\poly{P}_1$ defined by induction as follows:
\begin{enumerate}[{\bf i)}]
\item the canonical projections $x_i^{\sort{s}} : \sort{s} \fl \sort{s}_i$, for $1\leq i \leq \size{\sort{s}}$ are terms, called \emph{variables},
\item for all terms $f:\sort{s} \fl \ob{r}$ and $f':\sort{s} \fl \ob{r'}$ in $\freetha{\poly{P}}_1$, there exists a unique $1$-cell $\pair{f,f'} : \sort{s} \fl \ob{r}\ob{r'}$, called the \emph{pairing} of  terms $f,f'$, such that $x_1^{\ob{rr'}}\pair{f,f'} = f$ and $x_2^{\ob{rr'}}\pair{f,f'} = f'$,
\item for every operation $\varphi : \sort{r} \fl \ob{s}$ in $\poly{P}_1$, $\sort{s}$ in $\poly{P}_0^\ast$ and terms $f_i : \sort{s} \fl \sort{r}_i$ in $\freetha{\poly{P}}_1$ for $1\leq i \leq \size{\sort{r}}$, there is a term $\varphi\pair{f_1,\ldots,f_{\size{\sort{r}}}} : \sort{s} \fl \ob{s}$.
\end{enumerate}
We define the \emph{size} of a term $f$ as the minimal number, denoted by $\size{f}$, of operations used in its definition. The composition of terms  $f$ and $g$ is denoted by concatenation $fg$. For all $0$-cells $\sort{s},\sort{s}\ob{'}$ in $\freetha{\poly{P}}_1$, we denote by $\id_{\sort{s}}$ the identity $1$-cell on a $0$-cell $\sort{s}$, we denote by $e_{\sort{s}}$ the \emph{eraser} $1$-cell defined as the unique $1$-cell from~$\sort{s}$ to the terminal $0$-cell $\ob{0}$.
We denote respectively by $x_{\sort{s}}^{\sort{s}\sort{s}\ob{'}}:\sort{s}\sort{s}\ob{'}\fl \sort{s}$ (resp. $x_{\sort{s}\ob{'}}^{\sort{s}\sort{s}\ob{'}}:\sort{s}\sort{s}\ob{'}\fl \sort{s}\ob{'}$) the canonical projections. Finally, we denote by $\tau_{\sort{s},\sort{s}\ob{'}}:\sort{s}\sort{s}\ob{'} \fl \sort{s}\ob{'}\sort{s}$ the \emph{exchange} $1$-cell defined by 
$\tau_{\sort{s},\sort{s}\ob{'}} = \pair{x_{\sort{s}\ob{'}}^{\sort{s}\sort{s}\ob{'}},x_{\sort{s}}^{\sort{s}\sort{s}\ob{'}}
}$.

\subsubsection{Two-dimensional cartesian polygraphs}
A \emph{cartesian $2$-polygraph} $\poly{P}$ is a data $(\poly{P}_0,\poly{P}_1,\poly{P}_2)$ made~of 
\begin{enumerate}[{\bf i)}]
\item a signature $(\poly{P}_0,\poly{P}_1)$,
\item a cellular extension of the free theory $\freetha{\poly{P}}_1$, that is a set $\poly{P}_2$ equipped with two maps 
\[
\xymatrix{
\freetha{\poly{P}}_1
&
\ar@<+0.5ex>[l]^-{\partial_1^+}
\ar@<-0.5ex>[l]_-{\partial_1^-}
\poly{P}_2
}
\]
satisfying the following \emph{globular conditions} $\partial_0^\mu\circ \partial_1^- = \partial_0^\mu\circ \partial_1^+$, for $\mu\in\{-,+\}$.
\end{enumerate}
In the sequel, by abuse of notation, we let $\poly{P}_i$ stand for the underlying of a polygraph~$\poly{P}$.
An element $A$ of $\poly{P}_2$ is called a \emph{rule} with \emph{source} $\partial_1^-(A)$ and target $\partial_1^+(A)$, denoted respectively by $A_-$ and $A_+$. The globular conditions impose that a rule relates terms of same arity and coarity, and it will be pictured as follows:
\[
\xymatrix@C=4.5em{
\sort{s}
	\ar@/^3ex/ [r] ^-{A_-} ^{}="src"
	\ar@/_3ex/ [r] _-{A_+} ^{}="tgt"
	\ar@2 "src"!<0pt,-10pt>;"tgt"!<0pt,10pt> ^-{\;A}
&
\ob{r}
}
\quad\text{with}\quad
\sort{s} = \partial_0^-(A_-)=\partial_0^-(A_+),
\quad
\ob{r} = \partial_0^+(A_-)=\partial_0^+(A_+).
\]

\subsubsection{Two-dimensional theories}
Recall that a \emph{$2$-category} is a category enriched in categories. Explicitly, a $2$-category is a data $\mathcal{C}$ made of a set $\mathcal{C}_0$, whose elements are called the \emph{$0$-cells} of $\mathcal{C}$, and, for all $0$-cells $x,y$ of $\mathcal{C}$, a category $\mathcal{C}(x, y)$, whose $0$-cells and $1$-cells are respectively the \emph{$1$-cells} and \emph{$2$-cells} from $x$ to $y$ of $\mathcal{C}$. This data is equipped with a functor 
\[
\star_0^{x,y,z} : \mathcal{C}(x,y) \times \mathcal{C}(y, z) \fl \mathcal{C}(x, z),
\]
for all $0$-cells $x,y,z$ of $\mathcal{C}$, and a specified $0$-cell $\id_x$ of the category $\mathcal{C}(x, x)$. The composition $\star_0$ is associative, and the identities are local units for the composition. For $f_1\in \mathcal{C}(x,y)$ and $f_2\in \mathcal{C}(y,z)$, we write $f_1\star_0 f_2$ instead of $f_1\star_0^{x,y,z} f_2$. For $2$-cells $f_1,g_1$ in $\mathcal{C}(x,y)$ such that $(f_1)_+=(g_1)_-$, we denote by $f_1\star_1 g_1$ their composition along a $1$-cell from $x$ to $y$. The compositions $\star_0$ and $\star_1$ satisfy the \emph{exchange law}:
\[
(f_1\star_0 f_2) \star_1 (g_1\star_0 g_2) = (f_1\star_1 g_1) \star_0 (f_2\star_1 g_2),
\]
for all composable $2$-cells $f_i,g_i$ in $\mathcal{C}$. 

Recall that a \emph{$2$-theory} on a set of sorts $\poly{P}_0$ is a $2$-category with the additional following cartesian structure:
\begin{enumerate}[{\bf i)}]
\item it has a terminal $0$-cell \ob{0}, that is for every $0$-cell $\sort{s}$ there exists a unique \emph{eraser $1$-cell} $e_{\sort{s}} : \sort{s} \fl \ob{0}$, and the identity $2$-cell is the unique endo-$2$-cell on an eraser,
\item it has products, that is for all $0$-cells $\sort{r},\sort{r}\ob{'}$ there is a product $0$-cell $\sort{r}\sort{r}\ob{'}$ and $1$-cells $x_{\sort{r}}^{\sort{r}\sort{r}\ob{'}}: \sort{r}\sort{r}\ob{'} \fl \sort{r}$ and $x_{\sort{r}\ob{'}}^{\sort{r}\sort{r}\ob{'}}: \sort{r}\sort{r}\ob{'} \fl \sort{r}\ob{'}$ satisfying the following two conditions:
\begin{itemize}
\item for all $1$-cells $f_1:\sort{s} \fl \sort{r}$ and $f_2:\sort{s} \fl \sort{r}\ob{'}$, there exists a unique \emph{pairing $1$-cell} $\pair{f_1,f_2} : \sort{s} \fl \sort{r}\sort{r}\ob{'}$, such that $x_{\sort{r}}^{\sort{r}\sort{r}\ob{'}}\pair{f_1,f_2} =f_1$, and $x_{\sort{r}\ob{'}}^{\sort{r}\sort{r}\ob{'}}\pair{f_1,f_2} =f_2$,
\item for all $2$-cells $a_i : f_i \dfl f_i'$, $i=1,2$, there exists a unique $2$-cell $\pair{a_1,a_2} : \pair{f_1,f_2} \dfl \pair{f_1',f_2'}$. 
For $1$-cells $f_1,\ldots,f_k$, we will abbreviate $\pair{\id_{f_1},\ldots,\id_{f_k}}$ to $\pair{f_1,\ldots,f_k}$.
\end{itemize}
\end{enumerate}

A \emph{$(2,1)$-theory} is a $2$-theory whose every $2$-cell is invertible with respect to the $\star_1$-composition, \emph{i.e.}, every $2$-cell $a$ has an inverse $a^- : a_+ \dfl a_-$ 
satisfying the relations $ a \star_1 a^- = \id_{a_-}$ and~$a^- \star_1 a = \id_{a_+}$.

\subsubsection{Free $2$-theories}
\label{SS:Free2Theory}
We denote by $\freetha{\poly{P}}_2$ the free $2$-theory generated by a cartesian $2$-polygraph $\poly{P}$.
Its underlying $1$-category is the free theory $\freetha{\poly{P}}_1$ generated by the signature $(\poly{P}_0,\poly{P}_1)$. 
Its $2$-cells are defined inductively as follows:
\begin{enumerate}[{\bf i)}]
\item for all $2$-cell $A: f \dfl g$ in $\poly{P}_2$ and $1$-cell $h$ in $\freetha{\poly{P}}_1$, there is a $2$-cell $A h: fh \dfl gh$ in $\freetha{\poly{P}}_2$,
\item for all $2$-cells $a,b$ in $\freetha{\poly{P}}_2$, there is a $2$-cell $\pair{ a, b} : \pair{a_-, b_-} \dfl \pair{a_+,b_+}$ in $\freetha{\poly{P}}_2$,
\item for every $2$-cell $a$ in $\freetha{\poly{P}}_2$, there is a $2$-cell in $\freetha{\poly{P}}_2$ of the form $\Gamma[a] : \Gamma[a_-] \dfl \Gamma[a_+]$, where $\Gamma$ denotes a \emph{context} of the form:
\[
\Gamma := f\pair{f_1,\ldots,\boxx_j,\ldots,f_k} : \sort{s} \fl \ob{r},
\]
where $f_i:\sort{s} \fl \sort{r_i}$ and $f: \sort{r} \fl \ob{r}$ are $1$-cells of $\freetha{\poly{P}}_1$, and $\boxx_j$ is the $j$-th element of the pairing.
\item these $2$-cells are submitted to the following exchange relations 
\begin{equation*}
\label{E:exchangeRelation2cells}
 f \pair{f_1,...,a,...,f_j,...,f_k} \star_1 f \pair{f_1,...,f_i,...,b,...,f_k} = f \pair{f_1,...,f_i,...,b,...,f_k}\star_1 f \pair{f_1,a,...,f_j,...,f_k}
\end{equation*}
where $f_i: \sort{s} \fl \sort{r_i}$ and $f: \sort{r} \fl \ob{r}$ are $1$-cells in $\freetha{\poly{P}}_1$, and $a,b$ are $2$-cells in $\freetha{\poly{P}}_2$.
We will denote by $f\pair{f_1,...,a,...,b,...,f_k}$ the $2$-cell defined above.
\item The $\star_1$-composition of $2$-cells in $\poly{P}_2$ is given by sequential composition. 
\end{enumerate}
The source and target maps $\st^\pm_1$ extend to $\freetha{\poly{P}}_2$ and we denote $a_-$ and $a_+$ for $\st_1^-(a)$ and $\st_1^+(a)$ respectively.

The \emph{free $(2,1)$-theory} generated by $\poly{P}$, denoted by $\tck{\poly{P}}_2$, is constructed as the $2$-theory generated by cells of $\poly{P}$ and formal inverses of the $2$-cells of $\freetha{\poly{P}}_2$, and submitted to the relations $ a \star_1 a^- = \id_{a_-}$ and~$a^- \star_1 a = \id_{a_+}$, for every $2$-cell $a$. We define the congruence relation on $\freetha{\poly{P}}_1$ by $f \equiv_{\poly{P}} g$ if there is a $2$-cell of $\tck{\poly{P}}_2$ with source $f$ and target $g$. The \emph{theory presented} by $\poly{P}$ is the algebraic theory, denoted by $\cl{\poly{P}}$, and defined as the quotient of the free theory $\freetha{\poly{P}}_1$ by the congruence $\equiv_{\poly{P}}$.

\subsubsection{Ground terms}
Let $\poly{P}$ be a cartesian $2$-polygraph.
A \emph{ground term} in the free theory $\freetha{\poly{P}}_1$ is a term with source $\ob{0}$.
A $2$-cell $a$ in the free theory $\freetha{\poly{P}}_2$ is called \emph{ground} when $a_-$ is a ground term. Finally, a context 
$f\pair{f_1,\ldots,\boxx_j,\ldots f_{\size{\sort{r}}}}$ is called \emph{ground} when all the $f_i$ are ground terms.

\subsubsection{Rewriting properties of cartesian polygraphs}
The contexts can be composed in a natural way, and we will denote by $\Gamma \: \Gamma'[\boxx] := \Gamma[\Gamma'[\boxx]]$ the composition of contexts $\Gamma$ and $\Gamma'$. We define a \emph{multi-context} (of arity $2$) as
\[
\Delta[\boxx_i , \boxx_j] := f \pair{f_1 , \ldots, \boxx_i, \ldots, \boxx_j, \ldots, f_k},
\]
where the $f_k: \sort{s} \fl \sort{r_k}$ and $f: \sort{r} \fl \ob{r}$ are $1$-cells in $\freetha{\poly{P}}_1$, and $\boxx_i$ (resp. $\boxx_j$) has to be filled by a $1$-cell $g_i: \sort{s} \fl \sort{r_i}$ (resp. $g_j: \sort{\ob{s}} \fl \sort{\ob{r}_j}$).

A $2$-cell of the form $\Gamma[A h]$, where $\Gamma$ is a context, $h$ is a $1$-cell in $\freetha{\poly{P}}_1$ and $A$ is a rule in $\poly{P}_2$ is called a \emph{rewriting step} of $\poly{P}$. We consider the ARS $(\freetha{\poly{P}}_1,\poly{P}_{\text{stp}})$ where $\poly{P}_{\text{stp}}$ is the cellular extension made of rewriting steps of $\poly{P}$, whose source and target maps extend the ones of $\poly{P}$. We say that $\poly{P}$ is \emph{terminating} (resp. \emph{quasi-terminating}, \emph{confluent}) if the ARS $(\freetha{\poly{P}_1},\poly{P}_{\text{stp}})$ is so.
If $\poly{P}'$ is a cartesian $2$-polygraph with the same signature as $\poly{P}$, we say that $\poly{P}$ is confluent modulo $\poly{P}'$ if the ARS $(\freetha{\poly{P}_1},\poly{P}_{\text{stp}})$ is confluent modulo~$\equiv_{\poly{P}'}$.

For the sake of readability, we will denote terms and rewriting rules of cartesian polygraphs as in term rewriting theory, \cite{Terese03}. The canonical projection $x_i^{\sort{s}} : \sort{s} \fl \sort{s}_i$, for $1\leq i \leq |\sort{s}|$ is identified to the "\emph{variable}" $x_i$. A $1$-cell $f: \sort{s} \fl \ob{r}$, is denoted by $f(x_1,\ldots,x_{|\sort{s}|})$, and a rule $A : f \dfl g$ with $f,g: \sort{s} \fl {\ob r}$ will be denoted by
\[
A_{x_1,\ldots,x_{|\sort{s}|}} : f(x_1,\ldots,x_{|\sort{s}|}) \dfl g(x_1,\ldots,x_{|\sort{s}|}).
\]

\subsection{Algebraic examples}

\subsubsection{Magmas}
\label{SSS:AssociativeCommutativeMagmas}
Denote by $\poly{Mag}$ the cartesian $2$-polygraph, where  $\poly{Mag}_0:=\{\ob{1}\}$, $\poly{Mag}_1:=\{\mu : \ob{2} \fl \ob{1}\}$, and $\poly{Mag}_2$ is empty.
Denote by $\poly{Ass}$ the cartesian $2$-polygraph, where $\poly{Ass}_1 = \poly{Mag}_1$ and with a unique generating $2$-cell:
\begin{eqn}{equation}
\label{E:AssociativityRule}
\asso{\mu}_{x,y,z} \: : \: \mu (\mu (x,y) ,z ) \dfl \mu (x, \mu(y,z)).
\end{eqn}
Denote by $\poly{AC}$ the cartesian $2$-polygraph, where $\poly{AC}_1 = \poly{Mag}_1$, and $\poly{AC}_2$ is the disjoint union $ \poly{Ass}_2 \sqcup \{ \comm{\mu} \} $ with
\begin{eqn}{equation}
\label{E:CommutativityRule}
\comm{\mu} \: : \: \mu (x,y) \dfl \mu(y,x),
\end{eqn}
that corresponds to the rule $\comm{\mu} :\mu \tau \dfl \mu$, where $\tau$ is the exchanging operator defined in \eqref{SSS:SignatureAndTerms}. Note that the cartesian polygraph $\poly{AC}$ is not terminating, and that the rule $\comm{\mu}$ can not be oriented in a terminating way. As a consequence, for cartesian $2$-polygraphs whose set of rules contains commutativity and associativity for some operation, we will chose to work modulo the polygraph $\poly{AC}$.

The polygraphs $\poly{Mag}$, $\poly{Ass}$, and $\poly{AC}$ will be sometimes denoted by $\Mag{\mu}$, $\Ass{\mu}$, and $\AC{\mu}$ to refer to the label of the operation.

\subsubsection{Monoids}
\label{Ex:TheoryMonoids}
Denote by $\poly{Mon}$, or $\Monn{\mu,e}$, the cartesian $2$-polygraph with $\poly{Mon}_0:=\{\ob{1}\}$, $\poly{Mon}_1:=\Ass{\mu}_1 \sqcup \{e : \ob{0} \fl \ob{1}\}$, and $\poly{Mon}_2:=\Ass{\mu}_2\sqcup \{\neutr{\mu}_l,\neutr{\mu}_r \}$, where
\begin{eqn}{equation}
\label{E:NeutralElementMonoid}
\neutr{\mu}_l \: : \: \mu (e,x) \dfl x, \quad\text{and}\quad \neutr{\mu}_r \: : \:  \mu(x,e) \dfl x.
\end{eqn}
The presented theory $\cl{\poly{Mon}}$ is the theory of monoids. We also define the cartesian polygraph $\poly{CMon}$, with same $0$-cells and $1$-cells, and $\poly{CMon}_2 := \Monn{\mu,e}_2 \sqcup \{ \comm{\mu} \}$, where $\comm{\mu}$ is the commutativity $2$-cell~\eqref{E:CommutativityRule}. 

\subsubsection{Groups}
\label{SSS:PresentationOfGroups}
Denote by $\poly{Grp}$, or $\Grpp{\mu,e,\iota}$, the cartesian $2$-polygraph, where $\poly{Grp}_0:=\{\ob{1}\}$, $\poly{Grp}_1 := \Monn{\mu,e}_1 \sqcup \{ \iota : \ob{1} \fl \ob{1} \}$, and $\poly{Grp}_2 := \Monn{\mu,e}_2 \sqcup \{ \inv{\mu, \iota}_l , \inv{\mu, \iota}_r \}$, with
\begin{eqn}{equation}
\label{E:InverseRelGroup}
\inv{\mu, \iota}_l \: : \: \mu ( \iota(x),x ) \dfl e, \quad\text{and}\quad \inv{\mu, \iota}_r \: : \: \mu( x, \iota(x)) \dfl e.
\end{eqn}
The presented theory $\cl{\poly{Grp}}$ is the theory of groups. Following \cite{Hullot1980ACO}, the set of generating $2$-cells
\begin{align*}
\neutr{\mu}_l, \quad \neutr{\mu}_r, \quad \inv{\mu, \iota}_l, \quad \inv{\mu, \iota}_r, \quad \grp{\mu, \iota}_1 \: : \: \iota (e) \dfl e, \quad \grp{\mu, \iota}_2 \: : \: \iota(\mu(x,y)) \dfl \mu (\iota(y), \iota(x)),
\\
\grp{\mu, \iota}_3 \: : \: \iota (\iota(x)) \dfl x, \quad \grp{\mu, \iota}_4 \ : \: \mu(x,\mu(\iota(x),y)) \dfl y, \quad \grp{\mu, \iota}_5 \: : \: \mu ( \iota(x) , \mu (x,y)) \dfl y,
\end{align*}
defines a polygraph, denoted by $\widetilde{\poly{Grp}}$, that is convergent modulo $\Ass{\mu}$, and presents the theory~$\cl{\poly{Grp}}$.

\subsubsection{Abelian groups}
Denote by $\poly{Ab}$, or $\Abb{\mu,e,\iota}(\ob{1})$, the cartesian $2$-polygraph, where $\poly{Ab}_0:=\{\ob{1}\}$, $\poly{Ab}_1 = \Grpp{\mu,e,\iota}_1 $ and $\poly{Ab}_2 = \Grpp{\mu,e,\iota}_2 \sqcup \{ \comm{\mu} \}$, where $\comm{\mu}$ is the commutativity $2$-cell~\eqref{E:CommutativityRule}.

\subsubsection{Rings}
Denote by $\poly{Ring}$ the cartesian $2$-polygraph, where $\poly{Ring}_0:=\{\ob{1}\}$, 
\[
\poly{Ring}_1 = \Abb{+,0,-}_1 \sqcup \Monn{\,\cdot\,,1}_1
, \quad\text{and}\quad
\poly{Ring}_2 = \Abb{+,0,-}_2 \sqcup \Monn{\,\cdot\,,1}_2 \sqcup \{D_l,D_r\},
\]
with 
\begin{eqn}{equation}
\label{E:Distributivity}
D_l : x \cdot (y+z) \dfl x \cdot y + x \cdot z,
\qquad
D_r : (y+z) \cdot x \dfl y\cdot x + z \cdot x. 
\end{eqn}
Denote by $\poly{CRing}$, or $\poly{CRing}^{(+,0,-,\cdot,1)}(\ob{1})$, the cartesian $2$-polygraph with $\poly{CRing}_i = \poly{Ring}_i$, for $i=0,1$, and $\poly{CRing}_2 = \poly{Ring}_2 \sqcup \{ \comm{\cdot} \}$, where $\comm{\cdot}$ is the commutativity $2$-cell~\eqref{E:CommutativityRule}
The theory $\cl{\poly{CRing}}$ is the theory of commutative rings.
Following \cite{PetersonStickel81}, see also \cite{Hullot1980ACO}, the set of generating $2$-cells:
\begin{eqn}{equation}
\label{E:ConvergentPresentationCRingModAC}
\neutr{+}_r, \: \inv{+,-}_r, \: \grp{+,-}_1, \: \grp{+,-}_2, \: \grp{+,-}_3, \: D_r, \: R_1 \: : \: x \cdot 0 \dfl 0, \: R_2 \: : \: x \cdot (-y) \dfl - (x \cdot y), \: \neutr{\cdot}_r,
 \end{eqn}
defines a cartesian polygraph, that is convergent modulo $\AC{+} \sqcup \AC{\cdot}$, and presents the theory $\cl{\poly{CRing}}$.

\subsubsection{Modules over a commutative ring}
\label{SS:ConvergentPresentationRModModuloAC}
Denote by $\poly{Mod}$ the cartesian $2$-polygraph defined as follows. We set~$\poly{Mod}_0 = \{\ob{m},\ob{r}\}$, $\poly{Mod}_1 = \poly{CRing}^{(+,0,-,\cdot,1)}(\ob{r})_1 \sqcup \Abb{\oplus,0^{\oplus},\iota}(\ob{m})_1 \sqcup \{ \eta : \ob{r}\ob{m} \fl \ob{m}\}$, and we will denote $\eta(\lambda,x) = \lambda.x$, for $\lambda$ and $x$ of type $\ob{r}$ and $\ob{m}$ respectively.
We set
\[
\poly{Mod}_2 = \poly{CRing}^{(+,0,-,\cdot,1)}(\ob{r})_2 \sqcup \Abb{\oplus,0^{\oplus},\iota}(\ob{m})_2 \sqcup \{ M_1, M_2, M_3, M_4\},
\]
with
\begin{align*}
M_1 :  \lambda . (\mu . x)\dfl  (\lambda \cdot \mu).x,
\qquad
M_2 : 1.x \dfl x,\hspace{2.5cm}\\
M_3 : \lambda . (x \oplus y) \dfl (\lambda .x)\oplus (\lambda .y),
\qquad
M_4 : \lambda .x \oplus \mu .x \dfl (\lambda + \mu).x
\end{align*}

Following \cite{Hullot1980ACO}, the $2$-cells in~\eqref{E:ConvergentPresentationCRingModAC} together with the following set of $2$-cells
\begin{eqn}{align}
\label{E:ConvergentPresentationRMod}
M_1,\; M_2, \; M_3, \; M_4, \;\; N_1:x \oplus 0^{\oplus} \dfl x, \;\; N_2 : x \oplus (\lambda.x) \dfl (1+\lambda).x,\nonumber \\
N_3 : x \oplus x \dfl (1 + 1).x, \;\;N_4 : x . 0^\oplus \dfl 0^\oplus, \;\; N_5 : 0. x \dfl 0^\oplus, \;\; N_6 : \iota(x) \dfl (- 1).x,
\end{eqn}
gives a convergent presentation of the theory of modules over a commutative ring modulo the cartesian polygraph $\AC{+} \sqcup \AC{\cdot}$. This presentation can be summed up in the following set of rules:   
\begin{align*}
& x + 0 \dfl x & (\text{ring}_1) \qquad \qquad &  x + (-x) \dfl 0 \qquad  & (\text{ring}_2) \\
& - 0 \dfl 0 & (\text{ring}_3) \qquad \qquad &  - (-x) \dfl x \qquad & (\text{ring}_4) \\
& - (x+y) \dfl (-x) + (-y) & (\text{ring}_5) \qquad \qquad  & x \cdot (y + z) \dfl x \cdot y + x \cdot z \ \qquad & (\text{ring}_6) \\
& x \cdot 0 \dfl 0 & (\text{ring}_7) \qquad \qquad & x \cdot (-y) \dfl - (x \cdot y) \qquad & (\text{ring}_8) \\
& 1 \cdot x \dfl x & (\text{ring}_9) \qquad \qquad & a \oplus 0^{\oplus} \dfl a \qquad & (\text{mod}_1) \\
& x . (y . a) \dfl (x \cdot y) . a & (\text{mod}_2) \qquad \qquad & 1 . a \dfl a \qquad & (\text{mod}_3) \\
& x . a \oplus y . a \dfl (x + y) . a & (\text{mod}_4) \qquad \qquad & x . (a \oplus b) \dfl (x . a) \oplus (y . b) \qquad & (\text{mod}_5) \\
& a \oplus (r . a) \dfl (1+r) . a & (\text{mod}_6) \qquad \qquad & a \oplus a \dfl (1+1) . a \qquad & (\text{mod}_7) \\
& x . 0^{\oplus} \dfl 0^{\oplus} & (\text{mod}_8) \qquad \qquad & 0 . a \dfl 0^{\oplus} \qquad & (\text{mod}_{9}) \\
& I(a) \dfl (-1) . a  & (\text{mod}_{10}) \qquad \qquad & &
\end{align*}
Let us denote by $\poly{Mod}'_2$ the set containing the $2$-cells~\eqref{E:ConvergentPresentationCRingModAC} and \eqref{E:ConvergentPresentationRMod}.
We denote by $\poly{Mod}^{\textsf{c}}$ the cartesian $2$-polygraph $(\poly{Mod}_0, \poly{Mod}_1, \poly{Mod}'_2 \sqcup \AC{+} \sqcup \AC{\cdot})$. It also presents the theory $\cl{\poly{Mod}}$ of modules over a commutative ring.

\section{Algebraic polygraphs modulo}
\label{S:AlgebraicPolygraphsModulo}

In this section we introduce the notion of algebraic polygraphs, defined by cellular extensions on ground terms over a signature endowed with constants, and the notion of algebraic polygraphs modulo. We refer the reader to \cite{DupontMalbos18b} for a categorical formulation of the constructions given in this section. 

\subsection{Algebraic polygraphs}
\label{S:AlgebraicPolygraphs}

\subsubsection{Algebraic polygraphs}
An \emph{algebraic polygraph} is a data $(\poly{P},Q,R)$ made of
\begin{enumerate}[{\bf i)}]
\item a cartesian $2$-polygraph $\poly{P}$, 
\item a cellular extension $Q$ of $\poly{P}_0$ whose elements have source $\ob{0}$, and called \emph{constants},
\item a cellular extension $R$ of the sub-theory of the free theory $\freetha{(\poly{P}_0,\poly{P}_1\sqcup Q)}$ made of all ground terms, denoted by $\poly{P}_1\langle Q \rangle$.
\end{enumerate}

We have a decomposition 
\[ 
\poly{P}_1 \langle Q \rangle = \bigsqcup\limits_{s \in \poly{P}_0} \poly{P}_1 \langle Q \rangle_s, 
\]
where $\poly{P}_1 \langle Q \rangle_s$ contains the ground terms of coarity $s$, hence   the cellular extension $R$ is also indexed by the sorts of $\poly{P}_0$, so that it defines a 
family $(\poly{P}_1 \langle Q \rangle_s,R_{s})_{s\in \poly{P}_0}$ of ARSs.

\subsubsection{Rewriting properties of algebraic polygraphs}
Let $\Pr = (\poly{P},Q,R)$ be an algebraic polygraph. A \emph{$R$-rewriting step} is a ground $2$-cell in the free $2$-theory $\freetha{R}$ generated by $(\poly{P}_0,\poly{P}_1 \sqcup Q, R)$ of the form
\[
\Gamma[A] : \Gamma[f] \dfl \Gamma[g],
\]  
where $A : f \dfl g $ is a rule in $R$, and $\Gamma$ is a ground context. We denote by $R_{\text{stp}}$ the cellular extension made of $R$-rewriting steps of $\Pr$, whose source and target maps extend the ones of $R$. We say that $\Pr$ is \emph{terminating} (resp. \emph{quasi-terminating}, \emph{confluent}) if the ARS $(\poly{P}_1 \langle Q \rangle, R_{\text{stp}})$ is so. A \emph{$R$-rewriting path} is a finite or infinite sequence $a = a_1  \star_1 \ldots \star_1 a_k \star_1 \ldots$ of $R$-rewriting steps~$a_i$. 
The \emph{length} of a finite $R$-rewriting path $a$, denoted by $\ell(a)$, is the number of $R$-rewriting steps that it contains.

\smallskip
The cellular extension $\poly{P}_2$ of $\freetha{\poly{P}}_1$ extends to a cellular extension of the free $1$-theory $\freetha{(\poly{P}_1 \sqcup Q)}$. We denote by $\poly{P}_2 \langle Q \rangle$ the set of \emph{ground $2$-cells on $Q$} of the free $2$-theory generated by the $2$-polygraph $(\poly{P}_0,\poly{P}_1 \sqcup Q, \poly{P}_2)$. The data $(\poly{P},Q,\poly{P}_2 \langle Q \rangle)$ defines an algebraic polygraph. Two $1$-cells $f,g$ in $\poly{P}_1 \langle Q \rangle$ are \emph{algebraically equivalent} with respect to $\poly{P}$, and we denote $f \equiv_{\poly{P}_2 \langle Q \rangle} g$, if there exists a $2$-cell in $\tck{\poly{P}_2 \langle Q \rangle}$ with source $f$ and target $g$.

\smallskip
Let $\poly{P}'=(\poly{P}_0,\poly{P}_1,\poly{P}'_2)$ be a cartesian $2$-polygraph with the same signature as $\poly{P}$. We say that $\Pr$ is \emph{confluent modulo} the algebraic polygraph $(\poly{P}',Q,\poly{P}'_2 \langle Q \rangle)$ if the ARS $(\poly{P}_1 \langle Q \rangle, R_{\text{stp}})$ is confluent modulo $\equiv_{\poly{P}'_2 \langle Q \rangle}$.
The algebraic polygraph $(\poly{P},Q,\poly{P}_2 \langle Q \rangle)$ shares the rewriting properties of the polygraph $\poly{P}$. In particular, if $\poly{P}$ is terminating (resp. quasi-terminating, confluent), then so is $(\poly{P},Q,\poly{P}_2 \langle Q \rangle)$. Moreover, if $\poly{P}$ is confluent modulo $\poly{P}'$, then $(\poly{P},Q,\poly{P}_2 \langle Q \rangle)$ is confluent modulo $(\poly{P}',Q,\poly{P}'_2 \langle Q \rangle)$.

\subsubsection{Positive reduction strategies}
\label{SSS:PositiveStragegies} 
Denote by $\cl{\poly{P}\langle Q \rangle}$ the quotient of the theory $ \poly{P}_1\langle Q \rangle$ by the congruence relation $\equiv_{\poly{P}_2 \langle Q \rangle}$. In~\eqref{SS:AlgebraicRewrSystem}, we will consider rewriting with respect to a quotient algebraic system on $\cl{\poly{P} \langle Q \rangle}$ whose rules are the projections of the rules of $R$. Rewriting properties of this latter depend on $\poly{P}$.
In many situations, if we consider projections of all the $R$-rewriting steps we lose termination in the quotient rewriting system. This is the case when the algebraic theory is equipped with inverse operators, such as theories $\cl{\poly{Mod}}$ and $\cl{\poly{Grp}}$.
To prevent this, we need to select admissible $R$-rewriting steps compatible with~$\poly{P}$ using the following notion of strategy. 

\smallskip
Let $\pi : \poly{P}_1\langle Q \rangle \fl \cl{\poly{P}\langle Q\rangle}
$ be the canonical projection. We define a \emph{positive strategy} $\sigma$ as a map that associates to every $\cl{f} \in \cl{\poly{P} \langle Q \rangle}$ a non-empty subset $\sigma(\cl{f})$ of $\pi^{-1}(\cl{f})$. A $R$-rewriting step $a$ is called \emph{$\sigma$-positive} if $a_-$ belongs to $\sigma(\pi(a_-))$, and a $R$-rewriting path is called \emph{$\sigma$-positive} if every of its rewriting steps is positive.

\smallskip
In most cases, a positive strategy is defined uniformly with respect to $\poly{P}$ as follows. Suppose that $\poly{P}$ has a decomposition $\poly{P}_2 = \poly{P}'_2 \sqcup \poly{P}''_2$, where $\poly{P}'_2$ is terminating and confluent modulo $\poly{P}''_2$.  For every $1$-cell $\cl{f}$ in $\cl{\poly{P}\langle Q \rangle}$, we set
\[
\sigma(\cl{f}) = \bigsqcup\limits_{f \in \pi^{-1}(\cl{f})} NF(f,\poly{P}'_2),
\]
where $NF(f,\poly{P}'_2)$ is the set of normal forms of $f \in \poly{P}_1 \langle Q \rangle$ with respect to $\poly{P}'_2$. 
By confluence of $\poly{P}'_2$ modulo $\poly{P}''_2$, we deduce from \cite[Lemma 2.6]{Huet80} that any two elements of $\sigma(\cl{f})$ are congruent modulo $\poly{P}''_2$.

\subsubsection{Remarks}
\label{R:RemarkStrategies}
In many algebraic rewriting contexts, we have $\Ass{\mu} \subseteq \poly{P}''_2$. For instance, in the case of algebraic polygraphs over $\Monn{\mu}$, the usual strategy is obtained with $\poly{P}'_2$ empty and $\poly{P}''_2 = \Ass{\mu}$. Hence, every $1$-cell in $\poly{P}_1 \langle Q \rangle$ is a normal form for the empty polygraph modulo $\Ass{\mu}$, and thus the positive strategy consists in taking all the congruence class. In the case of algebraic polygraphs over $\poly{Mod}$, we set $\poly{P}''_2=\AC{+} \sqcup \AC{\cdot}$, and $\poly{P}'_2$ is the convergent presentation of $\poly{Mod}'_2$ modulo AC given in \eqref{SS:ConvergentPresentationRModModuloAC}.

\subsubsection{Example}
\label{SSS:ExampleString}
Consider the cartesian polygraph $P = \poly{Mon}$, a set $Q$ of constants, and a cellular extension $R$ of $\poly{P}_1 \langle Q \rangle$ as follows:
\begin{eqn}{equation}
\label{E:AlgebraicPolygraphB3+}
Q = \{s,t : \ob{0} \fl \ob{1}\},
\qquad
R=\{\; A \: : \: \mu(\mu(s,t),s) \dfl \mu(t,\mu(s,t))\;\}.
\end{eqn}
This data defines an algebraic polygraph $(\poly{P},Q,R)$.
For example, if we consider the context $\Gamma = \mu(\mu(s,\boxx),t)$, the rule $A$ induces the following rewriting step
\[
\Gamma[A] : 
\mu(\mu(s,\mu(\mu(s,t),s)),t) \dfl \mu(\mu(s,\mu(t,\mu(s,t)),t).
\]
The set $\poly{P}_2 \langle Q \rangle$ is defined by the associativity relations on ground terms on the constants $s$ and $t$. For instance, $\poly{P}_2 \langle Q \rangle$ contains the following ground $2$-cell:
\[
\mu(\mu(s,t),s) \dfl \mu(s,\mu(t,s)).
\]
For this algebraic polygraph over $\poly{Mon}$, we consider the positive strategy as in~\eqref{R:RemarkStrategies} with $\poly{P}'_2 = \varnothing$ and $\poly{P}''_2 = \poly{Ass}$, so that for every $\cl{f} \in \cl{P \langle Q \rangle}$ we have $\sigma(\cl{f}) = \pi^{-1}(\cl{f})$. In other words, $\sigma(\cl{f})$ is the set of all representatives of $\cl{f}$ modulo associativity. For example, if $\cl{f} = sts$, then $\sigma(\cl{f}) = \{ \mu(s,\mu(t,s)), \: \mu(\mu(s,t),s) \}.$

\subsubsection{Example}
\label{SSS:ExampleInverses}
As aforementioned, for algebraic theories with inverse operators we need positive strategies $\sigma$ such that $\sigma(\cl{f}) \ne \pi^{-1}(\cl{f})$. Consider the cartesian polygraph $ \poly{P}= \poly{Grp}$,  and $Q,R$ as defined in \eqref{E:AlgebraicPolygraphB3+}. 
There is a $R$-rewriting step of the form 
\[ \mu(\mu(\mu(s,t),s),s^-) \dfl \mu(\mu(t,\mu(s,t)),s^-). \]
The left hand side being algebraically equivalent to $\mu(s,t)$, this rewriting step yields a reduction \linebreak  $st \dfl tsts^-$ in the quotient algebraic system on $\cl{P \langle Q \rangle}$ defined in \eqref{SS:AlgebraicRewrSystem}, so that the latter cannot be terminating. For this reason, we have to consider a positive strategy for which this $R$-rewriting step is not positive. In \eqref{SSS:PositiveStrategiesGroups}, we define a positive strategy for algebraic polygraphs over $\poly{Grp}$, that is not defined with respect to normal forms of $\poly{P}$ as done in \eqref{R:RemarkStrategies}.
 
Consider the cartesian polygraph $P = \poly{Mod}^{\textsf{c}}$, and cellular extensions $Q,R$ as follows:
\[
Q = \{x,y : \ob{0} \fl m \},
\qquad
R=\{\, A \: : \: x \dfl y \,\}.
\]
There is a $R$-rewriting step $a : x+(-x) \dfl x+(-y)$ that projects onto a reduction $0 \dfl x-y$ in the quotient algebraic system on $\cl{P \langle Q \rangle}$. In this case, we choose the positive strategy $\sigma$ defined in \eqref{R:RemarkStrategies}, where the positive rewriting steps are those whose source is a normal form with respect to $\poly{Mod}'_2$ modulo $\poly{AC}$. Since $x+(-x)$ is not a normal form with respect to the set of $2$-cells of $\poly{Mod}'_2$, the rewriting step $a$ is not $\sigma$-positive.

Finally, let us note that whenever we work with a cartesian $2$-polygraph $\poly{P}$ that admits an inverse operator~$\iota$ and a neutral operator~$e$, then for every algebraic polygraph $(\poly{P},Q,R)$ and every rule $A$ in $R$, there is a $R$-rewriting step
\[
e \dfl \mu(A_-,\iota(A_+)).
\]
In order to make the quotient algebraic rewriting system on $\cl{P \langle Q \rangle}$ terminating, we need to consider a strategy $\sigma$ such that the above $2$-cell is not positive. Hence, we cannot have $\sigma(\cl{f}) = \pi^{-1} (\cl{f})$.

\subsection{Algebraic polygraphs modulo}
\label{SS:AlgebraicPolygraphModulo}

\subsubsection{Algebraic polygraph modulo}
Let $(\poly{P},Q,R)$ be an algebraic polygraph.
We denote by $\PRP$ the cellular extension of the theory $\poly{P}_1 \langle Q \rangle$
made of triple $(e,a,e')$, where $e,e'$ are $2$-cells in $\poly{P}_2\langle Q\rangle^\top$, and $a$ is a $R$-rewriting step such that $e_+=a_-$ and $a_+=e'_-$. Such a triple, also denoted by $e \star_1 a \star_1 e'$, is called a \emph{$\PRP$-rule}, and pictured by
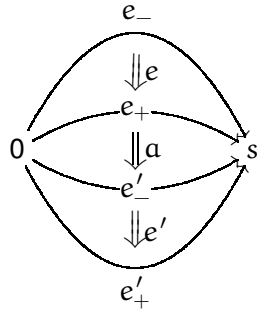
\[
\xymatrix@R=4em@C=7em{
\ob{0}
  \ar@/^9ex/[r] ^-{e_-} _-{}="1"
  \ar@/^3ex/[r] |-{e_+} ^-{}="2"
  \ar@/_3ex/[r] |-{e'_-} ^-{}="3"
  \ar@/_9ex/[r] _-{e'_+} _-{}="4"
&
s
	\ar@2 "1"!<0pt,-8pt>;"2"!<0pt,8pt> ^-{e}
	\ar@2 "2"!<0pt,-8pt>;"3"!<0pt,8pt> ^-{a}
	\ar@2 "3"!<0pt,-8pt>;"4"!<0pt,8pt> ^-{e'}
}
\]
Given a positive strategy $\sigma$ on $\Pr$, a rule $(e,a,e')$ is \emph{$\sigma$-positive} if $a$ is a $\sigma$-positive $R$-rewriting step. 
An \emph{algebraic polygraph modulo} is a data $\Pr=(\poly{P},Q,R,S)$ made of 
\begin{enumerate}[{\bf i)}]
\item an algebraic polygraph $(\poly{P},Q,R)$,
\item a cellular extension $S$ of $\poly{P}_1\langle Q \rangle$ such that $R \subseteq S \subseteq \PRP$. 
\end{enumerate}
We say that $\Pr$ is \emph{terminating} (resp. \emph{quasi-terminating}) if the algebraic polygraph $(\poly{P},Q,S)$ is terminating (resp. quasi-terminating).

\subsubsection{Example}
Let us consider the algebraic polygraph $(\poly{P},Q,R)$ defined in~\eqref{E:AlgebraicPolygraphB3+}, then the following composition gives a rewriting step in $\PRP$:
\[ 
(s \cdot (s \cdot (t \cdot s))) \cdot t \equiv_{\poly{P}_2 \langle Q \rangle} (s\cdot ((s\cdot t)\cdot s)) \cdot t \overset{\Gamma [A]}{\dfl} (s\cdot (t \cdot (s\cdot t)) \cdot t \equiv_{\poly{P}_2 \langle Q \rangle} ((s \cdot t) \cdot (s \cdot t)) \cdot t.
\]

\subsubsection{Quasi-normal forms}
\label{SSS:QuasiNormalForms}
Let $\Pr=(\poly{P},Q,R,S)$ be an algebraic polygraph modulo. A $1$-cell $f$ of $\poly{P}_1\langle Q \rangle$ is \emph{quasi-irreducible} if for every $S$-rewriting step $f \dfl g$ there exists a $S$-rewriting path from $g$ to $f$.
A \emph{quasi-normal form (with respect to $\Pr$)} of a $1$-cell $f$ in $\poly{P}_1\langle Q \rangle$ is a quasi-irreducible $1$-cell $\widetilde{f}$ of $\poly{P}_1\langle Q\rangle$ such that there exists a $S$-rewriting path from $f$ to $\widetilde{f}$. If $\Pr$ is quasi-terminating, every $1$-cell $f$ of $\poly{P}_1\langle Q \rangle$ admits at least a quasi-normal form, that is neither $S$-irreducible nor unique in general. A \emph{quasi-normal form strategy} is a map 
\[
s: \poly{P}_1\langle Q \rangle \fl \poly{P}_1\langle Q \rangle
\]
sending a $1$-cell $f$ on a chosen quasi-normal $\widetilde{f}$.

\subsection{Algebraic rewriting systems}
\label{SS:AlgebraicRewrSystem}

\subsubsection{Algebraic rewriting systems}
Let $\Pr=(\poly{P},Q,R,S)$ be an algebraic polygraph modulo.
A cellular extension $S$ of $\poly{P}_1 \langle Q\rangle$ extends to a cellular extension of the theory $\cl{\poly{P} \langle Q \rangle}$, with source $\cl{\st}_1^- := \pi \circ \st_1^-$, and target $\cl{\st}_1^+ := \pi \circ \st_1^+$.
An \emph{algebraic rewriting system} on $\Pr$ is a cellular extension $\cl{S}$ of $\cl{\poly{P}\langle Q\rangle}$ defined in such a way that the following diagram commutes
\[
\xymatrix@C=5em@R=2.5em{
&
S
\ar@<+0.5ex>[dl] ^-{\cl{\st}_1^-}
\ar@<-0.5ex>[dl] _-{\cl{\st}_1^+}
\ar[d] ^-{\pi'}
\\
\cl{\poly{P}\langle Q \rangle}
&
\cl{S}
\ar@<+0.5ex>[l] ^-{}
\ar@<-0.5ex>[l] _-{}
}
\]
where the map $\pi '$ assigns to a $S$-rule $e \star_1 a \star_1 e'$ an element $\cl{a}$ in $\cl{S}$ with source $\cl{a}_-$ and target $\cl{a}_+$.
Since $R \subseteq S \subseteq \PRP$, note that the quotient cellular extensions $\cl{R}$ and $\cl{S}$ coincide.

Given a positive strategy $\sigma$ on $\Pr$, let define $ 
\cl{S}^\sigma := \{ \cl{a}\in \cl{S} \; | \; \text{$a$ is a $\sigma$-positive $S$-rule} \}
$. A \emph{$\cl{S}$-rewriting step} (resp. $\cl{S}^\sigma$-rewriting step) is the quotient of a $S$-rewriting step (resp. $\sigma$-positive $S$-rewriting step) by the canonical projection $\pi$, that is a $2$-cell of the form $\cl{\Gamma[a]}: \cl{\Gamma[a_-]} \dfl \cl{\Gamma[a_+]}$,
where $\Gamma$ is a ground context of $\poly{P}_1\langle Q \rangle$ and $\Gamma[a]$ is a $S$-rewriting step (resp. $\sigma$-positive $S$-rewriting step). A \emph{$\cl{S}$-rewriting path} (resp. $\cl{S}^\sigma$-rewriting path) is a sequence of $\cl{S}$-rewriting steps (resp. $\cl{S}^\sigma$-rewriting steps).

\subsubsection{Examples}
\label{SSS:SRS}
A \emph{string rewriting system} (SRS) is an algebraic rewriting system on an algebraic polygraph modulo $(\poly{Mon},Q,R,S)$. The set $Q$ is the alphabet of the SRS, and the quotient of the cellular extension $R$ with respect to the congruence $\equiv_{\poly{Mon}_2 \langle Q \rangle}$ is the set of rules of the SRS. For instance, as a quotient of  the algebraic polygraph defined in~\eqref{E:AlgebraicPolygraphB3+}, we obtain the SRS 
\[ \langle s,t \: | \: sts \dfl tst \: \rangle, \]
that presents the monoid $B_3^+$ of braids on $3$ strands.

A \emph{linear rewriting system} (LRS) is an algebraic rewriting system on an algebraic polygraph modulo $(\poly{P},Q,R,S)$ such that $\poly{Mod}^{\textsf{c}} \subseteq \poly{P}$.

\section{Confluence of algebraic polygraphs modulo}
\label{S:ConfluenceAlgebraicPolygraphsModulo}

In this section we study confluence properties of algebraic polygraphs modulo with respect to positive strategies. Here $\Pr=(\poly{P},Q,R,S)$ denotes an algebraic polygraph modulo, and $\sigma$ a positive strategy on $\Pr$.

\subsection{Confluence modulo with respect to a positive strategy}

\subsubsection{Branchings in algebraic polygraphs modulo}
A \emph{$\sigma$-branching} of $\Pr$ is a triple $(a,e,b)$, where $a,b$ are $\sigma$-positive $2$-cells of $\freetha{S}$ and $e$ is a $2$-cell of 
$\tck{\poly{P}_2 \langle Q \rangle}$ as in the following diagram
\[
\xymatrix @R=1.25em @C=2.75em{
f
  \ar[r] ^-{a} 
  \ar[d] _-{e}
& 
f'
\\
g
  \ar[r] _-{b} 
&
g'
}
\]

In the rest of this article, for a better readability of the diagrams, the $2$-cells will be represented by simple arrows.  The \emph{source} of a $\sigma$-branching $(a,e,b)$ is the pair of $1$-cells $(f,g)$, where $f = a_- = e_-$, and $g = b_- = e_+$. When $b$ (resp. $a$) is an identity $2$-cell, the $\sigma$-branching is written $(a,e)$ (resp. $(e,b)$). When $e$ is an identity $2$-cell, the $\sigma$-branching is written $(a,b)$. A $\sigma$-branching $(a,e,b)$ is \emph{local} if $\ell(a) = \ell(b) + \ell(e) = 1$, that is it is either of the form $(a,e)$ or $(a,b)$.

A $\sigma$-branching $(a,e,b)$ is \emph{$\sigma$-confluent modulo} if there exist $\sigma$-positive $S$-rewriting paths $a',b'$, and a $2$-cell $e'$ in $\tck{\poly{P}_2 \langle Q \rangle}$ as in the following diagram:
\[
\xymatrix @R=1.25em @C=2.75em{
f
  \ar[r] ^-{a}
  \ar[d] _-{e}
&
f' 
  \ar@{.>}[r] ^-{a'} 
& 
h
  \ar[d] ^-{e'}
\\
g
  \ar[r] _-{b}
&
g'
  \ar@{.>}[r] _-{b'} 
&
h'
}
\]
The triple $(a',e',b')$ is called a \emph{$\sigma$-confluence modulo} of the branching $(a,e,b)$. We say that $\Pr$ is \emph{$\sigma$-confluent modulo} (resp. \emph{locally $\sigma$-confluent modulo}) if every $\sigma$-branching modulo (resp. local $\sigma$-branching modulo) is $\sigma$-confluent modulo.

\begin{remark} 
As noted in \cite{BachmairDershowitz89}, the algebraic polygraph $R$ is the polygraph for which it is the most difficult to reach $\sigma$-confluence modulo. Indeed, if $R$ is confluent modulo $\poly{P}$, then every algebraic polygraph modulo $(\poly{P},Q,R,S)$ is confluent modulo $\poly{P}$. For this reason, in many situations we relax by proving $\sigma$-confluence of $\PR$ or $\PRP$ modulo $\poly{P}$. 
In \cite{BachmairDershowitz89}, it is also noticed that when $\PRP$ is terminating, $\RP$ is confluent modulo~$\poly{P}$ if and only if $\PRP$ is confluent modulo $\poly{P}$, and in that case $\RP$ defines the same set of normal forms than~$\PRP$. As a consequence, we will either prove $\sigma$-confluence of $\RP$ and $\PRP$, leading to the same quotient algebraic rewriting system. Note finally that when $\PR \subseteq S \subseteq \PRP$, every local $\sigma$-branching modulo of the form $(a,e)$ is trivially $\sigma$-confluent modulo via the $\sigma$-confluence modulo $(\id_{a_-}, e^- \star_1 a,\id_{a_+})$.
\end{remark}

\subsubsection{Rewrite order on an algebraic polygraph modulo}
\label{SSS:RewriteOrder}
Denote by $\preccurlyeq_{\Pr}$ the relation on the $1$-cells of $\poly{P}_1 \langle Q \rangle$ defined, for all $1$-cells $f,g$ in $\poly{P}_1 \langle Q \rangle$, by $g \preccurlyeq_{\Pr} f$ if $f=g$ or $f$ $S$-rewrites into $g$.
The \emph{rewrite order} of~$\Pr$, denoted by $\prec_{\Pr}$, is the strict order on $\poly{P}_1 \langle Q \rangle$ defined by $g \prec_{\Pr} f$ if $g \preccurlyeq_{\Pr} f$ but not $f \preccurlyeq_{\Pr} g$. Note that when $\Pr$ is quasi-terminating, the relation $\preccurlyeq_{\Pr}$ does not define an order when there exists two $1$-cells which rewrite into each other, but the relation $\prec_{\Pr}$ is a well-founded strict order.

\subsubsection{Double induction principle}
Let us recall from Huet \cite{Huet80} the double induction principle, that we apply to quasi-terminating algebraic polygraphs modulo. From $\Pr$, we construct an auxiliary algebraic polygraph $\Paux := (\poly{P} \times \poly{P}, Q, S^{\text{db}})$, where $\poly{P}\times\poly{P}$ is the cartesian product of the polygraph $\poly{P}$ by itself, and the cellular extension $S^\text{db}$ on $(\poly{P} \times \poly{P})_1 \langle Q \rangle := P_1 \langle Q \rangle \times P_1 \langle Q \rangle$ contains a $2$-cell $(f,g) \dfl (f',g')$, for all $1$-cells $f,f',g,g'$ in $\poly{P}_1 \langle Q \rangle$ in any of the following situations: 
\begin{enumerate}[{\bf i)}]
\itemsep0em
\item there exists a $2$-cell $f \dfl f'$ in $\freetha{S}$ and $g=g'$;
\item there exists a $2$-cell $g \dfl g'$ in $\freetha{S}$ and $f=f'$;
\item there exist $2$-cells $f \dfl f'$ and $f \dfl g'$ in $\freetha{S}$;
\item there exist $2$-cells $g \dfl f'$ and $g \dfl g'$ in $\freetha{S}$;
\item  there exist $2$-cells $e_1,e_2,e_3$ in $\tck{\poly{P}_2 \langle Q \rangle}$, such that $\ell(e_1) > \ell(e_3)$, and as in the following diagram
\[ 
\xymatrix{ 
f \ar@2 [r] ^-{e_1} 
& 
g 
\ar@2 [r] ^-{e_2} 
& 
f' 
\ar@2 [r] ^-{e_3} & g'.
} 
\]
\end{enumerate}
As a consequence of the definition, if there exist $2$-cells $f \dfl f'$ and $g \dfl g'$ in $\freetha{S}$, then there is a $2$-cell $(f,g) \dfl (f',g')$ in $\Paux$ given by the composition $ (f,g) \dfl (f',g) \dfl (f',g')$.
Following \cite[Prop. 2.2]{Huet80}, if $\Pr$ is terminating, then so is $\Paux$. This result extends as follows: if $\Pr$ is quasi-terminating, then so is $\Paux$. Indeed, termination cycles that come from quasi-termination of $\Pr$ also appear in $\Paux$, and these are the only infinite rewriting paths that can arise.
In the sequel, we will prove rewriting results using double induction on a quasi-terminating algebraic polygraph modulo $\Pr$, consisting in using well-founded induction on the rewrite order $\prec_{\Paux}$ defined in \eqref{SSS:RewriteOrder}.

\begin{theorem}
\label{T:NewmanLemmaModulo}
Let $\Pr$ be a quasi-terminating algebraic polygraph modulo, and $\sigma$ be a positive strategy on $\Pr$. If $\Pr$ is locally $\sigma$-confluent modulo, then it is $\sigma$-confluent modulo.
\end{theorem}
\begin{proof}
Let $\Pr$ be locally $\sigma$-confluent modulo. We prove the result by well-founded induction with respect to the order $\prec_{\Paux}$. Let $(a,e,b)$ be a $\sigma$-branching modulo of $\Pr$ with source $(f,g)$. Suppose that for every $\sigma$-branching modulo $(a',e',b')$ with source $(f',g')$ such that there is a $2$-cell $ (f,g) \dfl (f',g')$ in $\freetha{(S^\text{db})}$, 
the $\sigma$-branching modulo $(a',e',b')$ is confluent modulo. We proceed in two steps.

\textbf{Step 1:} First, we prove that every $\sigma$-branching modulo $(a,e)$ with source $(f,g)$, where $a$ is a $\sigma$-positive $S$-rewriting step and $e$ is a $2$-cell in $\tck{\poly{P}_2 \langle Q \rangle}$, is $\sigma$-confluent modulo. We proceed by induction on $\ell(e) \geq 1$. If $\ell(e) = 1$, $(a,e)$ is local, hence it is $\sigma$-confluent modulo by assumption.
Now, assume that for $k \geq 1$, every $\sigma$-branching modulo $(a'',e'')$, such that $a''$ is a $\sigma$-positive $S$-rewriting step and $\ell(e'') = k$ is $\sigma$-confluent modulo, and consider a $\sigma$-branching modulo $(a,e)$ such that $\ell(e) = k+1$. We write $e = e_1 \star e_2$ with $e_1$ of length $1$. By local $\sigma$-confluence of the $\sigma$-branching modulo $(a,e_1)$, there exists a $\sigma$-confluence modulo $(a',e'_1,a_1)$ of this $\sigma$-branching. We write $a_1 = a_1^1 \star a_1^2$ with $a_1^1$ of length $1$ and $\ell(a_1^2) \geq 0$. By induction hypothesis on the $\sigma$-branching modulo $(a_1^1,e_2)$, there exists a $\sigma$-confluence modulo $(a'_1,e'_2,b)$ as in the following diagram:
\[
\xymatrix@R=2.25em @C=4em{ 
f \ar [d] _-{e_1} \ar [r] ^-{a} & f' \ar [r] ^-{a'} & f'' \ar [d] ^-{e'_1} \\
f_1 \ar [d] _-{\rotatebox{90}{=}} \ar [r] |-{a_1^1} ^-{}="2" & f'_1 \ar [d] ^-{\rotatebox{90}{=}} \ar [r] |-{a_1^2} & f''_1  \\
f_1 \ar [r] |-{a_1^1} _-{}="3" \ar [d] _ -{e_2} & f'_1 \ar [r] |-{a'_1} & f'_2 \ar [d] ^-{e'_2} \\
g \ar [rr] _-{b} ^-{}="1" & & g' 
\ar@2{} "1,2";"2,2" |{\textit{Local $\sigma$-conf mod}}
\ar@2{} "3,2";"1" |{\textit{Induction on $\ell(e)$}}
\ar@2{} "2";"3" |{=}
} 
\]
Now, since $\ell(e_1) = 1$ and $\ell(e_2) \geq 1$, we have the following rewriting path in $\Paux$:
\[
(f,g)  \dfl (f_1, g) \dfl (f_1,f_1) \dfl (f_1,f'_1) \dfl (f'_1,f'_1).
\]
We apply the double induction on the $\sigma$-branching $(a_1^2, a_1')$ with source $(f'_1,f'_1)$ to prove the existence of a $\sigma$-confluence modulo $(a_2,e_3,a'_2)$. By a similar argument, we use double induction on the $\sigma$-branchings modulo $(e'_1,a_2)$ and $(a'_2, e'_2)$ with respective sources $(f'',f_1'')$ and $(f'_2,g')$. Therefore, there exist $2$-cells $a''$,$a_3$, $a'_3$,  $b'$ in $\freetha{S}$ and $2$-cells $e''_1$, $e''_2$ in $\tck{\poly{P}_2 \langle Q \rangle}$ as in the following diagram:
\[
\xymatrix@R=2.25em@C=4.3em{ 
f \ar [d] _-{e_1} \ar [r] ^-{a} & f' \ar [r] ^-{a'} & f'' \ar [d] ^-{e'_1} \ar [rr]^-{a''} ^-{}="4" & & f''' \ar [d] ^-{e''_1} \\
f_1 \ar [d] _-{\rotatebox{90}{=}} \ar [r] |-{a_1^1} ^-{}="2" & f'_1 \ar [d] ^-{\rotatebox{90}{=}} \ar [r] |-{a_1^2} & f''_1 \ar [r] |-{a_2}  & h_1 \ar [r] |-{a_3} \ar [d] ^-{e_3} & h'_1  \\
f_1 \ar [r] |-{a_1^1} _ -{}="3" \ar [d] _ -{e_2} & f'_1 \ar [r] |-{a'_1} & f'_2 \ar [d] ^-{e'_2} \ar [r] |-{a'_2}  & h_2 \ar [r] |-{a'_3} & h'_2  \ar [d] ^-{e''_2} \\
g \ar [rr] _-{b} ^-{}="1" & & g' \ar [rr] _ -{b'} _-{}="5" & & g'' 
\ar@2{} "1,2";"2,2" |{\textit{Local $\sigma$-conf. mod}}
\ar@2{} "3,2";"1" |{\textit{Induction on $\ell(e)$}} 
\ar@2{} "2";"3" |{=}  
\ar@2{} "2,3";"3,3" |{\textit{Double Induction}}
\ar@2{}  "4";"2,4" |{\textit{Double Induction}}
\ar@2{}   "3,4";"5" |{\textit{Double Induction}} 
} 
\]   
Finally, we use once again double induction on the $\sigma$-branching modulo $(a_3,e_3,a'_3)$ of source $(h_1,h_2)$, satisfying $(h_1,h_2) \prec_{\Paux} (f,g)$, and repeat this process. Since the order $\prec_{\Paux}$ is well-founded, it terminates in finitely many steps until we reach quasi-normal forms $\widetilde{f}$ and $\widetilde{g}$ of $f$ and $g$ respectively. This yields the $\sigma$-confluence of the $\sigma$-branching $(a,e)$.

\medskip
\textbf{Step 2:}
Now, we prove that every $\sigma$-branching modulo $(a,e,b)$ with source $(f,g)$ is $\sigma$-confluent modulo. Suppose that  every $\sigma$-branching $(a',e',b')$ modulo with source $(f',g')$ such that there is a $2$-cell $(f,g) \dfl (f',g')$ in $\freetha{(S^\text{db})}$ is $\sigma$-confluent modulo. We use the proof scheme of \cite[Lemma 2.7]{Huet80}. Let us denote by $n := \ell(a)$ and $m:= \ell(b)$.
If both $m$ and $n$ are $0$, there is no branching modulo, so that we assume without loss of generality that $n > 0$. We write $a = a_1 \star_{1} a_2$ with $a_1$ of length $1$.

If $m = 0$, by Step 1 on the $\sigma$-branching modulo $(a_1,e)$, there exists a $\sigma$-confluence modulo $(a'_1,e',b')$ of this $\sigma$-branching. Then, we use double induction on the $\sigma$-branching modulo $(a_2,a'_1)$ with source $(f_1,f_1)$, since there is a rewriting path in $\freetha{(S^\text{db})}$ of the form 
\[ (f,g) \dfl (f,f) \dfl (f,f_1) \dfl (f,f_1). \]
There exist $\sigma$-positive $2$-cells $a'_2$, $a''_1$ in $\freetha{S}$ and a $2$-cell $e''$ in $\tck{\poly{P}_2 \langle Q \rangle}$ as follows:
\[
\xymatrix@R=2.25em @C=4.3em{ 
f \ar [r] ^-{a_1} ^-{}="2" \ar [d] _-{\rotatebox{90}{=}} & f_1 \ar [r] ^-{a_2} \ar [d] ^-{\rotatebox{90}{=}} & f_2 \ar [r]^-{a'_2} & f'_2 \ar [d] ^-{e''}\\
f \ar [d] _ -{e} \ar [r] |-{a_1} _ -{}="3" & f_1 \ar [r] |-{a'_1} & f_2 \ar [r] |-{a''_1} \ar [d] ^-{e'} & f'_2 \\
g \ar [rr] _ -{b'} ^-{}="1" & & g' & 
\ar@2{} "2,2";"1" |{\textit{Step 1}} \ar@2{} "2";"3" |{=}
\ar@2{} "1,3";"2,3" |{\textit{Double Induction}} } \]
We conclude the proof of this case with a similar argument as in Step 1, using repeated double inductions terminating after a finite number of steps by well-foundedness of the order $\prec_{\Paux}$.

Now, assume that $m > 0$ and write $b = b_1 \star_{1} b_2$ with $b_1$ of length $1$. By Step 1 on the $\sigma$-branching modulo $(a_1,e)$, there exists a $\sigma$-confluence modulo $(a'_1,e_1,c_1)$ of this $\sigma$-branching. We distinguish two cases whether $c_1$ is trivial or not. 

If $c_1$ is trivial, the $\sigma$-confluence of $(a,e,b)$ is obtained from the following diagram
\[
\xymatrix@R=2.25em@C=4.3em {
f \ar [d] _ -{\rotatebox{90}{=}} \ar [r] ^-{a_1} ^-{}="11" & f_1 \ar [d] ^-{\rotatebox{90}{=}} \ar [r] ^-{a_2} & f_2 \ar [rrr] ^-{a'_2} & & & f'_2 \ar [d] ^-{}    & \\
f \ar [r] |-{a_1} _-{}="12" \ar [d] _ -{e} & f_1 \ar [r] |-{a'_1} & f'_1 \ar [rr] |-{a_3} ^-{}="1" \ar [d] ^-{e'}  &  & f_3 \ar [d] ^-{e_1} \ar [r] ^-{a_4} ^-{}="8" & f_4 \ar [r] |-{a_5} & f_5 \ar [d] ^-{} \\
g \ar [d] _ -{\rotatebox{90}{=}} \ar [rr] |-{\id_g} ^-{}="2" & & g \ar [d] _-{\rotatebox{90}{=}} \ar [r] |-{b_1} ^-{}="5" & g'_1 \ar [d] ^-{\rotatebox{90}{=}} \ar [r] |-{b'_1} & g''_1 \ar [r] |-{b''_1} & h_1 \ar [d] ^-{} \ar [r] ^{b_3} & h_3 \\
g \ar [rr] _-{\id_g} _-{}="3" & & g \ar[r] |-{b_1} _-{}="6" & g'_1 \ar [r] _-{b_2} & g_2 \ar [r] _ -{b'_2} & h_2 &  
\ar@2{} "3,4";"2,4" |{\textit{Step 1}} 
\ar@2{} "2";"2,2" |{\textit{Step 1}}
\ar@2{} "11";"12" |{=} 
\ar@2{} "2";"3" |{=} 
\ar@2{} "5";"6" |{=} 
\ar@2{} "1,3";"2,3"!<75pt,0pt> |{\textit{Double Induction}}
\ar@2{} "2,6";"3,6" |{\textit{Double Induction}}
\ar@2{} "3,5";"4,5" |{\textit{Double Induction}}
} \]
where the $\sigma$-branchings modulo $(a_1,e)$ and $(b_1,e')$ are $\sigma$-confluent modulo by Step 1, and double induction applies on the $\sigma$-branchings $(a_2,a'_1 \star_{1} a_3)$, $(b'_1,b_2)$ and $(a_4,e_1,b''_1)$ of respective sources $(f_1,f_1)$, $(g'_1,g'_1)$ and $(f_3,g''_1)$ which are all strictly smaller than $(f,g)$ for $\prec_{\Paux}$. We then reach a $\sigma$-confluence modulo of the $\sigma$-branching modulo $(a,e,b)$ similarly using repeated double inductions.
\medskip

If $c_1$ is not trivial, write $c_1 = c_1^1 \star_{1} c_1^2$ with $c_1^1$ of length $1$. The $\sigma$-confluence of the $\sigma$-branching modulo $(a,e,b)$ is obtained from the following diagram:

\[ \xymatrix@R=2.25em@C=5.3em {
f \ar [d] _-{\rotatebox{90}{=}} \ar [r] ^-{a_1} ^-{}="1" & f_1 \ar [d] ^-{\rotatebox{90}{=}}  \ar [r] ^-{a_2} & f_2 \ar [r] ^-{a'_2} & f'_2 \ar [d] ^-{} & \\
f \ar [d] _ -{e} \ar [r] |-{a_1} _-{}="2" & f_1 \ar [r] |-{a'_1} & f'_1 \ar [d] ^-{} \ar [r]|-{a_3} & f_3 \ar [r] |-{a_4} & f_4 \ar [d] ^-{} \\
g \ar [d] _-{\rotatebox{90}{=}} \ar [r] |-{c_1^1} ^-{}="5" & g_1 \ar [d] ^-{\rotatebox{90}{=}} \ar [r] |-{c_1^2} & h_1 \ar [r] |-{c_2} & h_2 \ar [d] ^-{} \ar [r] |-{c'_2} & h'_2 \\
g \ar [d] _-{\rotatebox{90}{=}} \ar [r] |-{c_1^1} _-{}="6" &  g_1 \ar [r] |-{c'_1} & h'_1 \ar [d] ^-{} \ar [r] |-{c_3} & h_3 \ar [r]  |-{c'_3} & h'_3 \ar [d] ^-{} \\
g  \ar [d] _-{\rotatebox{90}{=}} \ar [r] |-{b_1} ^-{}="3" & g' \ar [d] ^-{\rotatebox{90}{=}} \ar [r] |-{b'_1} & g'_1 \ar [r] |-{b'_2} & g'_2 \ar [r] |-{b'_3} \ar [d] ^-{} & g'_3 \\
g \ar [r] _-{b_1} _-{}="4" & g' \ar [r] _-{b_2} & g_2 \ar [r] _-{b_3} & g_3 & 
\ar@2{} "1";"2" |{=} 
\ar@2{} "3";"4" |{=}  
\ar@2{} "5";"6" |{=}
\ar@2{} "2,2";"3,2" |{\textit{Step 1}} 
\ar@2{} "4,2";"5,2" |{\textit{Local $\sigma$-conf mod}}
\ar@2{} "1,3";"2,3" |{\textit{Double Induction}}
\ar@2{} "2,4";"3,4" |{\textit{Double Induction}}
\ar@2{} "3,3";"4,3" |{\textit{Double Induction}}
\ar@2{} "5,4";"4,4" |{\textit{Double Induction}}
\ar@2{} "5,3";"6,3" |{\textit{Double Induction}}
 } 
 \]
where the $\sigma$-branching modulo $(a_1,e)$ is confluent modulo by Step 1, the $\sigma$-branching modulo $(c_1^1,b_1)$ is $\sigma$-confluent by local $\sigma$-confluence modulo, and we check that double induction applies on the $\sigma$-branchings $(a_2,a'_1)$, $(c_1^2,c'_1)$, $(b'_1,b_2)$, $(a_3,c_2)$ and $(c_3,b'_2)$ of respective sources $(f_1,f_1)$, $(g_1,g_1)$, $(g',g')$ and $(f'_1,h_1)$ and $(h'_1,g'_1)$ which are all strictly smaller than $(f,g)$ for $\prec_{\Paux}$. Similarly, we can repeat inductions to reach a $\sigma$-confluence modulo of $(a,e,b)$.
\end{proof}

\subsection{Critical $\sigma$-branchings modulo}

\subsubsection{Classification of local $\sigma$-branchings}
\label{SSS:ConfluenceTerminationAlgebraicPolygraphsModulo}
The local $\sigma$-branchings modulo of $\Pr$ can be classified in the following families: 
\begin{enumerate}[{\bf i)}]
\item \emph{trivial} $\sigma$-branchings of the form
\begin{equation*}
\xymatrix@R=1.25em @C=3.5em{
\Gamma[a_-] \ar[d] _-{\rotatebox{90}{=}}
                \ar[r] ^-{\Gamma[a]} & \Gamma[a_+] \\
               \Gamma[a_-]
                \ar[r] _-{\Gamma[a]} & \Gamma[a_+]}
\end{equation*}
for all ground context $\Gamma$ and $\sigma$-positive $S$-rewriting step $a$.
\item \emph{orthogonal} $\sigma$-branchings modulo of the form
\[
\xymatrix@R=1.25em@C=4em{
\Delta[a_-,b_-] \ar[d] _-{\rotatebox{90}{=}}
                \ar[r] ^-{\Delta[a, b_-]} & \Delta[a_+,b_-] \\
              \Delta[a_-,b_-]
                 \ar[r] _-{\Delta[a_-, b]} & \Delta[a_-,b_+] }
\]
\[
\xymatrix@R=1.5em@C=4em{
\Delta[a_-,e_-] 
                 \ar[d] _-{\Delta[a_-, e]}
                \ar[r] ^-{\Delta[a, e_-]} & \Delta[a_+,e_-] \\
\Delta[a_-,e_+]  & }
\qquad
 \xymatrix@R=1.5em@C=3.5em{
\Delta'[e'_-,b_-] 
                 \ar[d] _-{\Delta'[e', b_-]}
                \ar[r] ^-{\Delta'[e'_-, b]} & \Delta'[e'_-,b_+] \\
\Delta'[e'_+,b_-]  & }       
\]
for all ground multi-contexts $\Delta$, $\Delta'$, $\sigma$-positive $S$-rewriting steps $a$, $b$, $c$, and $2$-cells $e$, $e'$ in $\tck{\poly{P}_2 \langle Q \rangle}$ of length $1$.
\item \emph{overlapping} $\sigma$-branchings are the remaining local $\sigma$-branchings. These branchings can be classified into two families: \emph{inclusion} $\sigma$-branchings of the form
\begin{equation*}
\label{E:InclusionIndependantbranchingsModulo}
\xymatrix@R=1.25em@C=4em{
\Gamma[a_-] \ar[d] _-{\rotatebox{90}{=}}
                \ar[r] ^-{\Gamma[a]} 
                & \Gamma[a_+] \\
 \Gamma [\Gamma'[b_-]] 
                 \ar[r] _-{\Gamma[\Gamma'[b]]}
                  & \Gamma[\Gamma'[b_+]]}
\end{equation*}
for all ground contexts $\Gamma$, $\Gamma'$, and $\sigma$-positive $S$-rewriting steps $a$, $b$, and \emph{regular overlapping} $\sigma$-branchings of the form
\begin{equation*}
\xymatrix@R=1.25em@C=4em{
\Gamma[a_-] \ar[d] _-{\rotatebox{90}{=}}
                \ar[r] ^-{\Gamma[a]} 
                & \Gamma[a_+] \\
\Lambda[b_-] 
                 \ar[r] _-{\Lambda[b]}
                  & \Lambda[b_+]}
\end{equation*}
for all ground contexts $\Gamma$, $\Lambda$, and $\sigma$-positive $S$-rewriting steps $a,b$ such that $(\Gamma[a],\Lambda[b])$ is not trivial, not orthogonal and not an inclusion branching. These branchings also admit their modulo counterpart, as in case $\mathbf{ii)}$, obtained by replacing the bottom $S$-rewriting step $b$ by a vertical $2$-cell $e$ in $\tck{\poly{P}_2 \langle Q \rangle}$ of length $1$.
\end{enumerate}

\subsubsection{Critical $\sigma$-branchings}
We define an order relation on $\sigma$-branchings modulo of $\Pr$ by setting $(a,e,b)\sqsubseteq (a',e',b')$ if there exists a ground context $\Gamma$ of $\poly{P}_1 \langle Q \rangle$ such that $a' = \Gamma[a]$, $e' = \Gamma[e]$ and $b' = \Gamma[b]$. A \emph{critical $\sigma$-branching modulo} is an overlapping $\sigma$-branching modulo that is minimal for the order relation $\sqsubseteq$. 

\subsubsection{Positive confluence}
We say that $\Pr$ is \emph{positively $\sigma$-confluent} if, for every $S$-rewriting step $a$, there exists $\widetilde{a_-} \in \sigma(a_-)$ and two $\sigma$-positive $S$-rewriting paths $a'$ $b'$ of length at most $1$ as in the following diagram
\[
\xymatrix@R=1.5em @C=2.25em{
\widetilde{a_-}
  \ar[rr] ^-{a'}
  \ar[d] _-{e}
&& f'
 \ar[d] ^-{\rotatebox{90}{=}}
\\
a_-
  \ar[r] _-{a}
& f
  \ar[r] _-{b'}
& f'
}
\]
where $e$ is a $2$-cell in $\tck{\poly{P}_2 \langle Q \rangle}$.
In that case, we say that $\sigma$ is a \emph{positive confluent strategy} for $\Pr$. 

\begin{proposition}
\label{P:TerminatingCriticalBranchingTheoremModulo}
Let $\Pr$ be a quasi-terminating algebraic polygraph modulo, and $\sigma$ be a positive strategy on $\Pr$. If $\Pr$ is positively $\sigma$-confluent, then it is locally $\sigma$-confluent modulo if, and only if, both of the following conditions are satisfied:
\begin{description}
\item[$\mathbf{a_0)}$] every critical $\sigma$-branching modulo $(a,b)$, where $a,b$ are $S$-rewriting steps, is $\sigma$-confluent modulo,
\item[$\mathbf{b_0)}$] every critical $\sigma$-branching modulo $(a,e)$, where $a$ is an $S$-rewriting step and $e$ is a $2$-cell in $\tck{\poly{P}_2 \langle Q \rangle}$ of length $1$, is $\sigma$-confluent modulo.
\end{description}
\end{proposition}

\begin{proof}  
One of the two implications is trivial. Suppose that condition $\mathbf{a_0)}$ holds, and prove that every local branching of the form $(a,b)$, where $a$, $b$ are $\sigma$-positive $S$-rewriting steps, is $\sigma$-confluence modulo. The proof that condition $\mathbf{b_0)}$ implies that every local branching of the form $(a,e)$, where $a$ is a $\sigma$-positive $S$-rewriting step and $e$ is a $2$-cell of $\tck{\poly{P}_2 \langle Q \rangle}$ of length $1$, is $\sigma$-confluent modulo is similar.

The proof is based on the analysis of all the possible cases of local $\sigma$-branchings modulo given in~\eqref{SSS:ConfluenceTerminationAlgebraicPolygraphsModulo}. Local trivial $\sigma$-branchings are always $\sigma$-confluent modulo. We consider a local orthogonal $\sigma$-branching modulo of the form 
\[ 
\xymatrix@R=1.7em@C=4em{
\Delta[a_-,b_-]
                 \ar[d] _-{\rotatebox{90}{=}}
                \ar[r] ^-{\Delta[a,b_-]} & \Delta[a_+,b_-]
               \\
 \Delta[a_-,b_-] \ar[r] _-{\Delta[a_-,b]}  &  \Delta[a_-,b_+] } 
 \]
 where $\Delta[a,b_-]$ and $\Delta[a_-,b]$ are $\sigma$-positive $S$-rewriting paths. There exist $2$-cells of $\freetha{S}$ as the dotted cells in the following diagram:
 \[ 
\xymatrix@R=1.7em@C=4em{
\Delta[a_-,b_-]
                 \ar[d] _-{\rotatebox{90}{=}}
                \ar[r] ^-{\Delta[a,b_-]} & \Delta[a_+,b_-] \ar@{.>} [r] ^-{\Delta[a_+,b]} & \Delta[a_+,b_+] \ar [d] ^-{\rotatebox{90}{=}}
               \\
 \Delta[a_-,b_-] \ar[r] _-{\Delta[a_-,b]}  &  \Delta[a_-,b_+]  \ar@{.>} [r] _-{\Delta[a,b_+]} & \Delta[a_+,b_+] } 
 \]
However, they are generally not $\sigma$-positive. Assume that they are both not $\sigma$-positive. By positive $\sigma$-confluence assumption, there exist a representative $1$-cell $\widetilde{\Delta[a_+,b_-]}$ (resp. $\widetilde{\Delta[a_-,b_+]}$) of $\Delta[a_+,b_-]$ (resp. $\Delta[a_-,b_+]$) in $\poly{P}_1 \langle Q \rangle$, $\sigma$-positive $S$-rewriting paths $c_1$, $c_2$, $d_1$, $d_2$, and $2$-cells $e_1$, $e_2$ in $\tck{\poly{P}_2 \langle Q \rangle}$ as in the following diagram:
\[ 
\xymatrix@R=1.7em@C=5em{ 
& \widetilde{\Delta[a_+,b_-]}  \ar [rr] ^-{c_1} & & \ar [d] ^-{\rotatebox{90}{=}} f
\\
\Delta[a_-,b_-]
                 \ar[d] _-{\rotatebox{90}{=}}
                \ar[r] ^-{\Delta[a,b_-]} & \Delta[a_+,b_-] \ar [u] ^-{e_1} \ar@{.>} [r] ^-{\Delta[a_+,b]} & \Delta[a_+,b_+] \ar [d] ^-{\rotatebox{90}{=}} \ar [r] |-{d_1} & f
               \\
 \Delta[a_-,b_-] \ar[r] _-{\Delta[a_-,b]}  &  \Delta[a_-,b_+] \ar@{.>} [r] _-{\Delta[a,b_+]} \ar [d] _-{e_2} & \Delta[a_+,b_+] \ar [r] |-{d_2} &  g   \ar [d] ^-{\rotatebox{90}{=}} \\
 & \widetilde{\Delta[a_-,b_+]} \ar [rr] _-{c_2} & & g 
 } 
 \] 
There is a rewriting path $(\Delta[a_-,b_-], \Delta[a_-,b_-]) \dfl (\Delta[a_+,a_+],\Delta[a_+,a_+])$ in $\freetha{(S^\text{db})}$ so that we apply double induction on the $\sigma$-branching modulo $(d_1,d_2)$. As a consequence, there exists a $\sigma$-confluence modulo $(d'_1,e',d'_2)$ of $(d_1,d_2)$. 
Then, we construct a $\sigma$-confluence modulo of $(\Delta[a,b_-], \Delta[a_-,b])$ by successive applications of induction as in the proof of Theorem \ref{T:NewmanLemmaModulo}. This process terminates since $\prec_{\Paux}$ is well-founded.

Let us now consider an overlapping $\sigma$-branching modulo of the form $(a,b)$, where $a$, $b$ are $\sigma$-positive $S$-rewriting steps. By definition, there exists a ground context $\Gamma$ of $\poly{P}_1 \langle Q \rangle$ and a critical $\sigma$-branching modulo $(a',b')$ such that $(a,b) = (\Gamma[a'], \Gamma[b'])$. Following condition $\mathbf{a_0)}$, the critical $\sigma$-branching $(a',b')$ is $\sigma$-confluent modulo, and there exists a $\sigma$-confluence modulo $(a'',e',b'')$ of this $\sigma$-branching. However, the $S$-rewriting paths $\Gamma[a'']$ and $\Gamma[b'']$ that would give a confluence modulo of $(a,b)$ are not necessarily $\sigma$-positive:
\[
\xymatrix@R=2em @C=3.5em{
u  \ar [d] _-{\rotatebox{90}{=}} \ar [r] ^-{a} & \ar@{.>} [r] ^-{\Gamma[a'']} & \ar [d] ^-{\Gamma[e']} \\
u \ar [r] _-{b} & \ar@{.>} [r] _-{\Gamma[b'']} & } \] 
Using positive $\sigma$-confluence of $S$, we are able to construct a $\sigma$-confluence modulo of the $\sigma$-branching modulo $(a,b)$ as in the previous case. 
\end{proof}

\subsubsection{Full positive strategy}
\label{R:NoTermination}
When all rewriting steps are positive, that is when $\sigma(\cl{f})= \pi^{-1} (\cl{f})$ for every $1$-cell $\cl{f}$ in $\cl{\poly{P}\langle Q\rangle}$, we say that $\sigma$ is a \emph{full positive strategy}. In that case, the quasi-termination assumption in Proposition~\ref{P:TerminatingCriticalBranchingTheoremModulo} is not needed to ensure local $\sigma$-confluence modulo from confluence of $\sigma$-critical branchings modulo. Indeed, the confluences represented by dotted arrows in the diagrams above are $\sigma$-positive. Moreover, the positive $\sigma$-confluence is always satisfied, by considering $a' = a$ and $b' = \id_{a_+}$.

\subsection{Algebraic critical branching lemma}

We now prove an algebraic critical branching lemma by quotienting the $S$-rewriting paths of Proposition~\ref{P:TerminatingCriticalBranchingTheoremModulo}.

\subsubsection{Critical branchings of algebraic polygraphs}
Let $\Ar$ be an algebraic rewriting system on $\Pr$. 
The \emph{critical branchings} of $\Ar$ are the projections of the critical $\sigma$-branchings modulo of $\Pr$ of the form $\mathbf{a}_0)$, that is pairs $(\cl{a},\cl{b})$ of $\cl{S}^\sigma$-rewriting steps such that there is a $\sigma$-branching modulo in $\Pr$ with source $(\widetilde{a_-},\widetilde{b_-})$. As a consequence of Proposition~\ref{P:TerminatingCriticalBranchingTheoremModulo}, we deduce the following result.

\begin{theorem}
\label{T:AlgebraicCriticalBranchingTheorem}
Let $\Pr=(\poly{P},Q,R,S)$ be an algebraic polygraph modulo with a positive confluent strategy $\sigma$. If $\PRP$ is quasi-terminating, then an algebraic rewriting system on $\Pr$ is locally confluent if, and only if, its critical branchings are confluent.
\end{theorem}

As an immediate consequence, we deduce the following critical branching lemma for algebraic polygraphs modulo.
\begin{corollary}
\label{C:ClassicalCriticalBranchingTheorem}
Let $\Pr$ be an algebraic polygraph modulo with a full positive strategy. Every algebraic rewriting system on $\Pr$ is locally confluent if, and only if, all its critical branchings are confluent.
\end{corollary}

\section{Examples of algebraic rewriting systems}

In this section, we apply the algebraic critical branching lemma to SRS, LRS, and group rewriting systems.

\subsection{String rewriting systems}
\label{SS:CriticalBranchingSRS}

\subsubsection{Critical branching lemma for string rewriting systems}
In \eqref{SSS:SRS} we show how to define a SRS as an algebraic rewriting system over the cartesian polygraph $\poly{Mon}$ given in~\eqref{Ex:TheoryMonoids}. In that case, Theorem~\ref{T:AlgebraicCriticalBranchingTheorem} is the following critical branching lemma for SRS as proved by Nivat, \cite{Nivat73}. 

\begin{theorem}
\label{T:CBSRS}
Let $\Pr$ be an algebraic polygraph modulo on the cartesian polygraph $\poly{Mon}$. Then an algebraic rewriting system on $\Pr$ is locally confluent if and only if its critical branchings are confluent.
\end{theorem}

In that case, the choice of positive strategy $\sigma$ making all the $2$-cells in $\freetha{S}$ be $\sigma$-positive implies that the positive $\sigma$-confluence is obvious. Moreover the quasi-terminating hypothesis is not required as explained in \eqref{R:NoTermination}.

\subsection{Linear rewriting systems}
\label{SS:CriticalBranchingLinearRewriting}

In this subsection, $\Pr=(\poly{P},Q,R,S)$ denotes an algebraic polygraph modulo, whose cartesian polygraph~$\poly{P}$ has an underlying linear structure, that is, $\poly{P}$ contains the cartesian polygraph $\poly{Mod}^{\textsf{c}}$.
We consider a decomposition of $\poly{P}$ as in~\eqref{SSS:PositiveStragegies},
with $\poly{P}''_2=\AC{+} \sqcup \AC{\cdot}$ and $\poly{P}'_2=\poly{Mod}^{\textsf{c}}$, and the positive strategy $\sigma$ on $\Pr$ of normal forms modulo $\AC{+} \sqcup \AC{\cdot}$ defined in~\eqref{SSS:PositiveStragegies}.

\subsubsection{Critical branching lemma for linear rewriting systems}
The algebraic polygraph $\PRP$ is never terminating. Indeed, because of the linear context, for every $R$-rule $a: f \dfl g$, we have a $\PRP$-rewriting step given by
\begin{eqn}{equation}
\label{E:LinearCycle}
\xymatrix{g \equiv_{\poly{P}} -f + (g+f) \ar@2 [rrr] ^-{-a + (g+f)} & & & -g + (g+f) \equiv_{\poly{P}} f} 
\end{eqn}
However, if the rewriting system  $\cl{S}^\sigma$ is terminating, then $\PRP$ is quasi-terminating, then as a consequence of Theorem~\ref{T:AlgebraicCriticalBranchingTheorem} we have

\begin{theorem}
\label{T:CBLlinear1}
Let $\Pr$ be a terminating algebraic polygraph modulo, whose cartesian polygraph has an underlying linear structure, and with a positive confluent strategy $\sigma$. 
Then an algebraic rewriting system on $\Pr$ is locally confluent if, and only if, its critical branchings are confluent.
\end{theorem}

Consider an algebraic rewriting system $\cl{S}$ on $\Pr$.
The positivity confluence of $S$ with respect to $\sigma$ implies the factorisation property of \cite[Lemma 3.1.3]{GuiraudHoffbeckMalbos19}, stating that every rewriting step $\cl{a}$ of $\cl{S}$ can be decomposed in the free $(2,1)$-theory on $\cl{S}$ as $\cl{a} = \cl{b} \star \cl{c}^{-1}$, where $\cl{b}$ and $\cl{c}$ are either positive rewriting steps of $\cl{S}^\sigma$ or identities, as in the following diagram:
\begin{eqn}{equation}
\label{E:FactorisationProperty}
\raisebox{0.8cm}{
\xymatrix@R=1.5em@C=2em{
& h & \\
f \ar@2@{.>}@/_1.5ex/ [rr] _-{\cl{a}} \ar@2@/^1.5ex/ [ur] ^-{\cl{b}}  & & g \ar@2@/_1.5ex/ [ul] _-{\cl{c}} 
}}
\end{eqn}
Note that if $\cl{a}$ is a rewriting step of $\cl{S}^\sigma$, this factorisation is trivial. When $\cl{a}$ is in $\cl{S}$ but not in $\cl{S}^\sigma$, that is $\cl{a}$ is a quotient of a non-$\sigma$-positive $S$-rewriting path, it states that $\cl{a}$ can be factorised using positive reductions. 
This proves the following critical branching criterion for linear algebraic rewriting systems.

\begin{theorem}
\label{T:CBLlinear2}
Let $\Pr$ be a terminating algebraic polygraph modulo, whose cartesian polygraph has an underlying linear structure, and satisfying the factorisation property~\eqref{E:FactorisationProperty}. 
Then an algebraic rewriting system on $\Pr$ is locally confluent if, and only if, its critical branchings are confluent.
\end{theorem}

\subsubsection{Left-monomial rewriting systems}
The rules of an algebraic rewriting system on $\Pr$ transform linear combinations of terms into linear combinations of terms. The system is called \emph{left-monomial} when the source of every rule is an element of $\poly{P}_1\langle Q \rangle$ that does not contain neither the operation $\oplus : \ob{mm} \fl \ob{m}$ nor $\eta : \ob{rm} \fl \ob{m}$ defined in~\eqref{SS:ConvergentPresentationRModModuloAC}. Equivalently, the source of any rule of the algebraic rewriting system is a \emph{monomial}.

For terminating left-monomial LRS, the local confluence is equivalent to the confluence of critical branchings, \cite[Thm. 4.3.2]{GuiraudHoffbeckMalbos19}. The proof of this criterion requires the factorisation property~\eqref{E:FactorisationProperty} that always holds in this context. 
We expect that in the left-monomial linear setting the positive confluence is equivalent to this property. But this remains an open problem, whose answer would explain the criterion for local confluence of LRS as a rewriting modulo result.

\subsection{Rewriting with inverses}
\label{SS:RewritingInverse}

We conclude these algebraic examples by presenting a notion of group rewriting system defined as an algebraic rewriting system.

\subsubsection{Rewriting in groups}
In group theory rewriting gives algorithmic methods for decision problems, such as the word/conjugacy/geodesic problems, \cite{LeChenadec84,LeChenadec86, DiekertDuncanMyasnikov10, DiekertDuncanMyasnikov12, Chouraqui09, Chouraqui11}. In most cases, the method consists in constructing a convergent presentation of the considered group. Note also that homological finiteness conditions for finite convergence of groups were introduced, \cite{CremannsOtto96}. Finally, algorithms to compute relations among relations (syzygies) for groups given by generators and relations were developed in \cite{HeyworthWendsley03}.
However, in all these works the presentations of the groups are interpreted by SRS, or by Gröbner bases, that present groups, or group rings, as monoids, or monoid rings, with axioms of inverses given explicitly in the set of rules. Namely, for a group $G$ presented by a set of generators $X$ and a set of relations $\Rr$, it is associated the following SRS:
\[
\langle \, Q \;|\; \, \eta_x : xx^- \fl 1,\; \eta^-_x : x^-x\fl 1, \rho_r : r \fl 1, \;\text{for $r\in \Rr$} \rangle.
\]

When solving decision problems, or computing homological invariants for groups, the rules $\eta_x$ and~$\eta^-_x$ make the problem more complicated uselessly. Indeed, these rules should not be considered as those defining the group. In this way, the notion of rewriting in groups is not algebraically well considered yet.

\subsubsection{Group rewriting systems}
Consider an algebraic polygraph modulo $\Pr=(\poly{P},Q,R,\PRP)$, where $\poly{P}=\widetilde{\poly{Grp}}$. 
The generating $1$-cells of $\poly{P}$ induce on $\cl{\poly{P}\langle Q \rangle}$ a structure of group isomorphic to the free group $F(Q)$ on $Q$.
Denote by $\poly{P}_1 \langle Q \rangle_{\text{red}}$ the set of reduced $1$-cells of $\poly{P}_1 \langle Q \rangle$ with respect to $\poly{P}_2 \langle Q \rangle$. 
A cellular extension $T$ of $\poly{P}_1\langle Q \rangle$ is called \emph{reduced} if, for every $A$ in $T$, the ground terms $A_-$ and $A_+$ belong to $\poly{P}_1 \langle Q \rangle_{red}$.

\begin{lemma}
There exists a unique reduced cellular extension $R_{red}$ of the theory $\poly{P}_1\langle Q \rangle$ such that the algebraic rewriting systems $\cl{R}$ and $\cl{R}_{red}$ on $\cl{\poly{P}\langle Q \rangle}$ coincide.
\end{lemma}
\begin{proof}
The $2$-cells of $R_{red}$ are obtained by reducing the sources and targets of $2$-cells of~$R$ with respect to $\poly{P}_2\langle Q \rangle$.
\end{proof}

From now on, we assume that the cellular extension $R$ is reduced.

\subsubsection{Positive strategies for reductions in groups}
\label{SSS:PositiveStrategiesGroups}
The free group $\cl{\poly{P}\langle Q \rangle}$ can be constructed as a quotient monoid. Indeed, consider the free monoid $(Q \sqcup Q^-)^\ast$ over the set $Q \sqcup Q^-$ of constants and their formal inverses, with $Q^-=\{ x^- \: \mid \: x \in Q \}$. Then, the group $\cl{\poly{P}\langle Q \rangle}$ is isomorphic, as a monoid, to the monoid generated by $Q \sqcup Q^-$ and submitted to the relations 
\begin{eqn}{equation}
\label{E:InverseRelations}
x x^- \to 1, \quad\text{and}\quad  x^- x \to 1, \quad\text{for every $x \in Q$}.
\end{eqn}
The relations~\eqref{E:InverseRelations} are convergent, and thus the elements of the group $\cl{\poly{P}\langle Q \rangle}$ are identified with normal forms of elements of $(Q \sqcup Q^-)^\ast$ with respect to these relations 

Let us fix a total order $\prec$ over $Q \sqcup Q^-$ such that for all $x,y\in Q$, $x\prec y$ implies $x^-\prec y^-$. Denote by $\precdeg$ the deglex order on the free monoid $(Q \sqcup Q^-)^\ast$ induced by the order $\prec$, that is for any $f, g \in (Q \sqcup Q^-)^\ast$, $f \precdeg g$ if $f$ is shorter than $g$ or they have the same length and $f$ is smaller than $g$ for the lexicographic order induced by $\prec$.

Every $1$-cell in $\poly{P}_1\langle Q \rangle$ can be written $f(\iota^{n_1}(x_1),\ldots,\iota^{n_k}(x_k))$, where $n_1, \dots, n_k \in \N$, $f$ is an element of $\freetha{\poly{P}}_1$, $x_1, \dots, x_k$ are constants of $Q$, $\iota$ is the inverse operation defined in \eqref{SSS:PresentationOfGroups}, and $\iota^0$ denotes the identity $1$-cell of the theory $\freetha{\poly{P}}_1$. Moreover, if each $n_i$ is chosen to be maximal, then $f$ is uniquely determined, and does not contain the operation $\iota$ in its leafs.
We define a map 
\[
\llbracket \; \rrbracket : \poly{P}_1 \langle Q \rangle \to (Q \sqcup Q^-)^\ast,
\]
that associates to every $1$-cell $f(\iota^{n_1}(x_1),\ldots,\iota^{n_k}(x_k))$ in $\poly{P}_1 \langle Q \rangle$, where the $n_i$'s are maximal as above, the word $x_1^{\varepsilon_1}\ldots x_k^{\varepsilon_k}$, where $\varepsilon_i=+$ if $n_i$ is even, and $\varepsilon_i=-$ if $n_i$ is odd.

Let us denote by $red(f)$ the normal form in $(Q \sqcup Q^-)^\ast$ of $\llbracket f \rrbracket$ with respect to relations \eqref{E:InverseRelations}. Let $\models$ be the order on $\poly{P}_1\langle Q\rangle$ defined by $f \models g$ if $red(f) \precdeg red(g)$.

We define a positive strategy for $\Pr$, by setting, for every $\cl{h} \in \cl{\poly{P}\langle Q \rangle}$, the set $\sigma(\cl{h})$ to be the subset of $\pi^{-1}(\cl{h})$ whose elements are of the form
$\mu(\mu(f,r_1^\varepsilon), g)$ and $\mu(f,\mu(r_1^\varepsilon, g))$, where
$f,g\in \poly{P}_1\langle Q\rangle_{\text{red}}$, $r_1 \to r_2 \in R$, $\varepsilon \in \{-,+ \}$, and such that
\[
\mu(\mu(f,r_2^\varepsilon), g) \models \mu(\mu(f,r_1^\varepsilon), g),
\]
where, for $i=1,2$, we let $r_i^\varepsilon:=r_i$ if $\varepsilon=+$, and $r_i^\varepsilon:=\iota(r_i)$ otherwise.

\begin{proposition}
For the positive strategy $\sigma$ defined above, the algebraic polygraph modulo  \linebreak $\Pr=(P,Q,R,\PRP)$ is positively $\sigma$-confluent. 
\end{proposition}
\begin{proof}
Let us introduce an auxiliary strategy $\sigma'$ for $\Pr$ by setting
  \begin{eqn}{equation}
    \label{E:InterPositiveStrategyGroups}
    \sigma'(\cl{h})=\left\{\Gamma[r_1]\in\pi^{-1}(\cl{h})\mid
    \text{$\Gamma$ is a context of $\poly{P}_1 \langle Q \rangle$}, \: r_1 \to r_2 \in R,  \: \text{s.t.} \:
    \Gamma[r_2] \models \Gamma[r_1] \right\},
  \end{eqn}
for every $\cl{h} \in \cl{\poly{P}\langle Q \rangle}$.
Prove that $\Pr$ is positively $\sigma'$-confluent.
For all rule $r_1\fl r_2$ in $R$ and ground context $\Gamma$ of $\poly{P}_1\langle Q \rangle$ such that
$\Gamma[r_2]\models \Gamma[r_1]$, the $\PRP$-rewriting step $\Gamma[r_1]\fl \Gamma[r_2]$ is $\sigma'$-positive.  Otherwise
$\Gamma[r_1] \models \Gamma[r_2]$, then the $\PRP$-rewriting step
  $\Gamma'[r_1] \to \Gamma'[r_2]$ is $\sigma'$-positive, where $\Gamma'[\square]=\Gamma[\mu(\mu(r_2,\iota(\square)),r_1)]$. Indeed, we have $\text{red}(\Gamma'(r_2))=\text{red}(\Gamma[r_1])  \precdeg \text{red}(\Gamma[r_2]) = \text{red}(\Gamma'(r_1))$. Moreover, $\Gamma[\mu(\mu(r_2,r_1^{-}),r_1)]$ and $\Gamma[\mu(\mu(r_2,r_2^{-}),r_1)]$ are equivalent with respect to $\equiv_{\poly{P}_2 \langle Q \rangle}$ to
  $\Gamma[r_2]$ and $\Gamma[r_1]$, respectively. Now, we show that every $\sigma'$-positive $\PRP$-rewriting step induces a $\sigma$-positive one.

Let us consider a $\sigma'$-positive $\PRP$-rewriting step $\Gamma[r]: \Gamma[r_1] \to \Gamma[r_2]$, 
let $n$ be the largest integer such that $\Gamma[r_1] = \Gamma_1 [\iota^n(r_1)]$ and $\Gamma_1$ is a (possibly empty) context. Denote by $\varepsilon := +$ if $n$ is even and $-$ if $n$ is odd, then $\iota^n(r_1)$ is equivalent to $r_1^\varepsilon$  modulo $\equiv_{\poly{P}_2\langle Q\rangle}$.

If $\Gamma_1$ is empty, then the $\PRP$-rewriting step is of the form $r_1^{\varepsilon} \to r_2^{\varepsilon}$. Since $\Gamma[r_2] \models \Gamma[r_1]$, then $r_2^{\varepsilon} \models r_1^{\varepsilon}$ and thus it is $\sigma$-positive.

Otherwise, $\Gamma_1 [r_1^\varepsilon]$ may be written either as $\mu(\mu(f', r_1^\varepsilon), g')$
or $\mu(f', \mu(r_1^\varepsilon, g'))$, where $f',g'$ are $1$-cells in $\poly{P}_1\langle Q\rangle$. Denote by $f:= \widehat{f'}$ and $g := \widehat{g'}$ be the normal forms of $f'$ and $g'$ with respect to $\poly{P}_2 \langle Q \rangle$.
Then $\Gamma_1 [r_1^\varepsilon]$ is equivalent modulo $\equiv_{\poly{P}_2 \langle Q \rangle}$ to
$\mu(\mu(f, r_1^\varepsilon), g)$ or $\mu(f, \mu(r_1^\varepsilon, g))$. Moreover, since $ red(f r_2^\varepsilon g) = red (\Gamma[r_2]) \precdeg red(\Gamma[r_1]) = red(f r_1^\varepsilon g)$,
the $\PRP$-rewriting step $f r_1^\varepsilon g \to f r_2^\varepsilon g$ is $\sigma$-positive, where $f r_i^\varepsilon g$ denotes either
$\mu(\mu(f, r_i^\varepsilon), g)$ or $\mu(f, \mu(r_i^\varepsilon, g))$.
\end{proof}

\subsubsection{Example}
Let us consider the algebraic polygraph modulo $(\poly{P},Q,R, \PRP)$, where $\poly{P}=\widetilde{\poly{Grp}}$, $Q = \{ s, \: t \}$ and $R = \{ \mu ( \mu(s,t),s) \dfl \mu (t, \mu(s,t))\}$. 
We consider the deglex order induced by the ordering $s>t>s^->t^-$.
The positive $\PRP$-rewriting steps are of the form 
\[ 
f \mu ( \mu(s,t),s) g \dfl f \mu (t, \mu(s,t)) g \quad \text{or} \quad f \mu( \mu(s^-,t^-),s^-) g \dfl f \mu (t^-, \mu(s^-,t^-)) g,
\]
where $f,g$ are reduced elements of $\poly{P}_1 \langle Q \rangle_{\text{red}}$, and the orientation is compatible with the order $\models$ as defined in~\eqref{SSS:PositiveStrategiesGroups}. For instance, there is a positive $\PRP$-rewriting step
\[ \mu ( \mu (\mu(s,t),s),t) \dfl \mu ( \mu (t, \mu(s,t)), t) \]
yielding a reduction $stst \dfl tstt$ in the free group $F(Q)$.

Now suppose that $f= tst$, $g = st$ and $\varepsilon = -1$. 
There is a $\sigma$-positive $\PRP$-reduction as follows:
\[ tst \mu(\mu(s^-,t^-),s^-) st \equiv_{\poly{P}} tst \mu(\mu(s^-,t^-),s^-) sts s^- \dfl tst \mu(t^-, \mu(s^-,t^-)) stss^- \equiv_{\poly{P}} \mu(s,t) \]
that gives a rewriting step $tsts^- \dfl st$ in the quotient.
There is a critical branching of $\PRP$ as follows:
\[ 
\xymatrix@R=1.5em@C=2em{
\mu(\mu(\mu(s,t),s), \mu(t,s)) \ar [r] \ar [d] & \mu(\mu(t,\mu(s,t)), \mu(t,s)) 
\\
\mu (\mu(s,t), \mu(\mu(s,t),s)) \ar [r] & \mu (\mu(s,t), \mu(t,\mu(s,t))) 
}
\]
that is not confluent modulo. It induces the following non confluent algebraic critical branching in the free group $F(Q)$
\[
\xymatrix@R=0.4em @C=3em{
& tstts \\
ststs \ar@/^2ex/ [ur]
\ar@/_2ex/ [dr] & \\
& sttst } 
\]

\section{Conclusion and perspectives}

In this article, we introduced the notion of algebraic rewriting systems as rewriting systems over algebraic theories. We studied algebraic contexts such as string, linear, and group rewriting. We formulated sufficient conditions to prove the critical branching lemma for algebraic rewriting systems. Our results lead us to formulate several perspectives:
\begin{enumerate}[$\bullet$]
\item In Section~\ref{SS:CriticalBranchingSRS}, we recovered the critical branching lemma for SRS with respect to a convergent presentation of the theory $\poly{Mon}$ and a positive strategy making all the reductions positive. This corresponds to the classical setting of SRS. One may wonder what happens if we consider another presentation of the theory $\poly{Mon}$ and another positive strategy. These choices define a paradigm of string rewriting. This raises the question of defining a notion of equivalence between paradigms of string rewriting.
\item For left-monomial LRS and Gröbner bases the critical branching lemma only requires termination. Theorem~\ref{T:CBLlinear2} proves that the factorisation property is also required. This property is always satisfied when we rewrite in left-monomial linear structures such as commutative or associative algebras. We expect that for left-monomial LRS, the factorisation property is equivalent to the positive confluence, and is always satisfied.
\item In Section~\ref{SS:RewritingInverse}, we defined a positive strategy to rewrite in a free group. We prove a critical branching lemma with respect to this strategy. However, we do not yet know an algorithm that computes the exhaustive list of critical branchings with respect to this strategy. The same algorithmic problem occurs for the computation of the critical branchings for LRS that are not left-monomial.
\item Another issue is to extend the algebraic critical branching lemma to higher-structures such as linear operads. Rewriting was defined on linear operads in terms of shuffle Gröbner bases by Dotsenko and Khoroshkin in  \cite{DotsenkoKhoroshkin10} and shuffle linear polygraphs by Malbos and Ren in \cite{MalbosRen20}. 
Algebraic polygraphs introduced in this article describe rewriting in one-dimensional algebraic structures, such as monoids, groups, modules, and algebras. We expect that our constructions can be extended to the setting of linear operads by considering algebraic polygraphs defined over a structure of cartesian $2$-polygraphs on shuffle trees.
\item Finally, another outlook is to extend the algebraic critical branching lemma to conditional rewriting systems in order to formalise the critical branching lemma for LRS defined over a field. The conditional rules are used to specify the rules depending on the invertibility of scalars in the field.
\end{enumerate}

\begin{small}
  \renewcommand{\refname}{\Large\textsc{References}}
  \bibliographystyle{plain}
\bibliography{biblioCURRENT}
\end{small}

\quad

\vfill

\begin{footnotesize}
\bigskip
\auteur{Cyrille Chenavier$^1$}{cyrille.chenavier@jku.at}
{Johannes Kepler University\\
Altenberger Straße 69\\
A-4040 Linz, Austria}

\bigskip
\auteur{Benjamin Dupont}{bdupont@math.univ-lyon1.fr}
{Univ Lyon, Universit\'e Claude Bernard Lyon 1\\
CNRS UMR 5208, Institut Camille Jordan\\
43 blvd. du 11 novembre 1918\\
F-69622 Villeurbanne cedex, France}

\bigskip
\auteur{Philippe Malbos}{malbos@math.univ-lyon1.fr}
{Univ Lyon, Universit\'e Claude Bernard Lyon 1\\
CNRS UMR 5208, Institut Camille Jordan\\
43 blvd. du 11 novembre 1918\\
F-69622 Villeurbanne cedex, France}
\end{footnotesize}

\vspace{1cm}

\begin{small}
--------------

$^1$ The author was supported by the Austrian  Science  Fund  (FWF):  P 32301. 

\vspace{0.6cm}

---\;\;\today\;\;-\;\;\hhmm\;\;---
\end{small}
\end{document}

%% file: macros.tex
\numberwithin{equation}{subsection}
\newcounter{fakecnt}[subsubsection]

\newcommand{\periodafter}[1]{\ifstrempty{#1}{}{#1.}}
\titleformat{\section}[block]{\scshape\filcenter\LARGE}{\thesection.}{.5em}{}
\titleformat{\subsection}[block]{\bfseries\filcenter\large}{\thesubsection.}{.5em}{\medskip}
\titleformat{\subsubsection}[runin]{\bfseries}{\thesubsubsection.}{.5em}{\periodafter}
\titlespacing{\subsubsection}{0pt}{\topsep}{.5em}

\newtheoremstyle{ntheorem}%
	{\topsep}{\topsep}{\itshape}{0pt}{\bfseries}{.}{.5em}%
	{\thmnumber{#2.\hspace{.5em}}\thmname{#1}\thmnote{ (#3)}}
	
\newtheoremstyle{ndefinition}%
	{\topsep}{\topsep}{\normalfont}{0pt}{\bfseries}{.}{.5em}%
	{\thmnumber{#2.\hspace{.5em}}\thmname{#1}\thmnote{ (#3)}}
	
\newtheoremstyle{nremark}%
	{\topsep}{\topsep}{\normalfont}{0pt}{\itshape}{.}{.5em}%
	{\thmnumber{}\thmname{#1}\thmnote{ (#3)}}

\theoremstyle{ntheorem}
  	\newtheorem{theorem}[subsubsection]{Theorem}
  	\newtheorem{proposition}[subsubsection]{Proposition}
	\newtheorem{lemma}[subsubsection]{Lemma}
  	\newtheorem{corollary}[subsubsection]{Corollary}

\theoremstyle{ndefinition}

	\newtheorem{remark}[subsubsection]{Remark}
	
\makeatletter
\def\@equationname{equation}
\newenvironment{eqn}[1]{%
    \def\mymathenvironmenttouse{#1}%
    \ifx\mymathenvironmenttouse\@equationname%
        \refstepcounter{subsubsection}%
    \else
        \patchcmd{\@arrayparboxrestore}{equation}{subsubsection}{}{}
        \patchcmd{\print@eqnum}{equation}{subsubsection}{}{}%
        \patchcmd{\incr@eqnum}{equation}{subsubsection}{}{}%
    \fi
    \csname\mymathenvironmenttouse\endcsname%
}{%
    \ifx\mymathenvironmenttouse\@equationname%
        \tag{\thesubsubsection}%
    \fi
    \csname end\mymathenvironmenttouse\endcsname%
}
\makeatother

\fancypagestyle{plain}{
\fancyhf{}\fancyfoot[LC,RC]{}
\fancyfoot[LE,RO]{$\thepage$}

}

\UseTips
\SelectTips{eu}{11}

\newdir{ >}{{}*!/-10pt/@{>}}
\newdir{ -}{{}*!/-10pt/@{}}
\newdir{> }{{}*!/+10pt/@{>}}

\makeatletter

\xyletcsnamecsname@{dir4{}}{dir{}}
\xydefcsname@{dir4{-}}{\line@ \quadruple@\xydashh@}
\xydefcsname@{dir4{.}}{\point@ \quadruple@\xydashh@}
\xydefcsname@{dir4{~}}{\squiggle@ \quadruple@\xybsqlh@}
\xydefcsname@{dir4{>}}{\Tttip@}
\xydefcsname@{dir4{<}}{\reverseDirection@\Tttip@}

\xydef@\quadruple@#1{%
	\edef\Drop@@{%
		\dimen@=#1\relax
		\dimen@=.5\dimen@
		\A@=-\sinDirection\dimen@
		\B@=\cosDirection\dimen@
		\setboxz@h{%
			\setbox2=\hbox{\kern3\A@\raise3\B@\copy\z@}%
			\dp2=\z@ \ht2=\z@ \wd2=\z@ \box2
			\setbox2=\hbox{\kern\A@\raise\B@\copy\z@}%
			\dp2=\z@ \ht2=\z@ \wd2=\z@ \box2
			\setbox2=\hbox{\kern-\A@\raise-\B@\copy\z@}%
			\dp2=\z@ \ht2=\z@ \wd2=\z@ \box2
			\setbox2=\hbox{\kern-3\A@\raise-3\B@ \noexpand\boxz@}%
			\dp2=\z@ \ht2=\z@ \wd2=\z@ \box2
		}%
		\ht\z@=\z@ \dp\z@=\z@ \wd\z@=\z@ \noexpand\styledboxz@
	}%
}

\xydef@\Tttip@{\kern2pt \vrule height2pt depth2pt width\z@
	\Tttip@@ \kern2pt \egroup
	\U@c=0pt \D@c=0pt \L@c=0pt \R@c=0pt \Edge@c={\circleEdge}%
	\def\Leftness@{.5}\def\Upness@{.5}%
	\def\Drop@@{\styledboxz@}\def\Connect@@{\straight@{\dottedSpread@\jot}}}
	
\xydef@\Tttip@@{%
	\dimen@=.25\dimen@
 	\B@=\cosDirection\dimen@
	\setboxz@h\bgroup\reverseDirection@\line@ \wdz@=\z@ \ht\z@=\z@ \dp\z@=\z@
	{\vDirection@(1,-1)\xydashl@ \xyatipfont\char\DirectionChar}%
	{\vDirection@(1,+1)\xydashl@ \xybtipfont\char\DirectionChar}%
}

\xydef@\ar@form{
	\ifx \space@\next \expandafter\DN@\space{\xyFN@\ar@form}%
	\else\ifx ^\next \DN@ ^{\xyFN@\ar@style}\edef\arvariant@@{\string^}%
	\else\ifx _\next \DN@ _{\xyFN@\ar@style}\edef\arvariant@@{\string_}%
	\else\ifx 0\next \DN@ 0{\xyFN@\ar@style}\def\arvariant@@{0}%
	\else\ifx 1\next \DN@ 1{\xyFN@\ar@style}\def\arvariant@@{1}%
	\else\ifx 2\next \DN@ 2{\xyFN@\ar@style}\def\arvariant@@{2}%
	\else\ifx 3\next \DN@ 3{\xyFN@\ar@style}\def\arvariant@@{3}%
	\else\ifx 4\next \DN@ 4{\xyFN@\ar@style}\def\arvariant@@{4}%
	\else\ifx \bgroup\next \let\next@=\ar@style
	\else\ifx [\next \DN@[##1]{\ar@modifiers{[##1]}}
	\else\ifx *\next \DN@ *{\ar@modifiers}%
	\else\addLT@\ifx\next \let\next@=\ar@slide
	\else\ifx /\next \let\next@=\ar@curveslash
	\else\ifx (\next \let\next@=\ar@curveinout 
	\else\addRQ@\ifx\next \addRQ@\DN@{\ar@curve@}%
	\else\addLQ@\ifx\next \addLQ@\DN@{\xyFN@\ar@curve}%
	\else\addDASH@\ifx\next \addDASH@\DN@{\defarstem@-\xyFN@\ar@}%
	\else\addEQ@\ifx\next \addEQ@\DN@{\def\arvariant@@{2}\defarstem@-\xyFN@\ar@}%
	\else\addDOT@\ifx\next \addDOT@\DN@{\defarstem@.\xyFN@\ar@}%
	\else\ifx :\next \DN@:{\def\arvariant@@{2}\defarstem@.\xyFN@\ar@}%
	\else\ifx ~\next \DN@~{\defarstem@~\xyFN@\ar@}%
	\else\ifx !\next \DN@!{\dasharstem@\xyFN@\ar@}%
	\else\ifx ?\next \DN@?{\ar@upsidedown\xyFN@\ar@}%
	\else \let\next@=\ar@error
	\fi\fi\fi\fi\fi\fi\fi\fi\fi\fi\fi\fi\fi\fi\fi\fi\fi\fi\fi\fi\fi\fi\fi \next@}

\makeatother

\newcommand{\fl}{\rightarrow}

\newcommand{\dfl}{\Rightarrow}

\newcommand{\qfl}{\xymatrix@1@C=10pt{\ar@4 [r] &}}

\newcommand{\cl}[1]{\overline{#1}}

\newcommand{\tck}[1]{#1^{\top}}

\DeclareMathOperator{\id}{Id}

\renewcommand{\phi}{\varphi}
\renewcommand{\epsilon}{\varepsilon}

\newcommand{\Nb}{\mathbb{N}}

\newcommand{\Ar}{\mathcal{A}}

\renewcommand{\Pr}{\mathcal{P}}

\newcommand{\Rr}{\mathcal{R}}

\newcommand{\ifthen}[2]{\ifthenelse{#1}{#2}{}}

